\documentclass[reqno]{amsart}
\usepackage{amsmath}
\usepackage{amsfonts}
\usepackage{amstext}
\usepackage{amsbsy}
\usepackage{amsopn}
\usepackage{amsxtra}
\usepackage{upref}
\usepackage{amsthm}
\usepackage{amsmath}
\usepackage{amssymb}
\usepackage{enumerate}
\usepackage{bbm}
\usepackage{hyperref}

\usepackage{enumitem}

\newtheorem{teo}{Theorem}[section]
\newtheorem{prop}[teo]{Proposition}
\newtheorem{lema}[teo]{Lemma}
\newtheorem{coro}[teo]{Corollary}
\newtheorem{rem}[teo]{Remark}
\newtheorem{defi}[teo]{Definition}
\newtheorem{eje}[teo]{Example}

\newtheorem{claim}{Claim}
\newtheorem*{claim*}{Claim}

 \usepackage{euscript}

\DeclareMathSymbol{\varnothing}{\mathord}{AMSb}{"3F}
\renewcommand{\emptyset}{\varnothing}

\def\cK{\EuScript{K}}

\def\R{{\mathbb R}}

\def\N{{\mathbb N}}

\def\M{{\mathcal M}}
\def\sm{\setminus}

\title{Recurrence and transience for suspension flows}
\date{\today}

\begin{thanks}
{ G.I. was partially  supported by  the Center of Dynamical Systems and Related Fields c\'odigo ACT1103 and by Proyecto Fondecyt 1110040.  T.J. wishes
to thank Proyecto Mecesup-0711 for funding his visit to PUC-Chile. M.T. would like to thank I. Melbourne and D. Thompson for useful comments.  He is also grateful for the support of Proyecto Fondecyt 1110040 for funding his visit to PUC-Chile and for partial support from NSF grant DMS 1109587.  All three authors thank the referees for their careful reading of the paper and useful suggestions.}
 \end{thanks}

\author{Godofredo Iommi} \address{Facultad de Matem\'aticas,
Pontificia Universidad Cat\'olica de Chile (PUC), Avenida Vicu\~na Mackenna 4860, Santiago, Chile}
\email{\href{mailto:giommi@mat.puc.cl}{giommi@mat.puc.cl}}
\urladdr{\url{http://www.mat.puc.cl/\textasciitilde giommi/}}
\author{Thomas Jordan} \address{The School of Mathematics, The University of Bristol, University Walk, Clifton, Bristol, BS8 1TW, UK}
\email{\href{mailto:Thomas.Jordan@bristol.ac.uk}{Thomas.Jordan@bristol.ac.uk}}
\urladdr{\url{http://www.maths.bris.ac.uk/~matmj/}}
\author{Mike Todd}
\address{Mike Todd\\ Mathematical Institute\\
University of St Andrews\\
North Haugh\\
St Andrews\\
KY16 9SS\\
Scotland} \email{\href{mailto:mjt20@st-andrews.ac.uk}{mjt20@st-andrews.ac.uk}}
\urladdr{\url{http://www.mcs.st-and.ac.uk/~miket/}}
\begin{document}

\begin{abstract}
We study the thermodynamic formalism for  suspension flows  over countable Markov shifts  with roof functions not necessarily bounded away from zero. We establish conditions to ensure the existence and uniqueness of equilibrium measures for regular potentials. We define the notions of recurrence and transience of a potential in this setting. We define the \emph{renewal flow}, which is a symbolic model for a class of flows with diverse recurrence features.  We study the corresponding thermodynamic formalism, establishing conditions for the existence of equilibrium measures and phase transitions.
Applications are given to suspension flows defined over interval maps having parabolic fixed points.
\end{abstract}

\maketitle

\section{Introduction}

In this paper we study suspension flows, that is a discrete dynamical system on the `base' along with a `roof' function which determines the time the flow takes to return to this base.  In particular we consider suspension flows over Markov shifts.  The ergodic theory of suspension flows with a Markov structure has been studied extensively in the context of Axiom A flows \cite {bo1, br, ra}, geodesic flows on surfaces of negative curvature \cite{mo, ar, Ser86}, and billiard flows \cite{bs1, bs2, bsc1}.  The main novelty here, aside from the facts that we develop a more general theory of thermodynamic formalism than in the works above and that we study countable Markov shifts, is that we do not assume that the roof function is uniformly bounded away from zero, which leads to significant technical difficulties.

Thermodynamic formalism for suspension flows over countable Markov shifts began with the work of Savchenko \cite{sav}. He gave a  definition of topological entropy in the case that the roof function depends only on the first coordinate, but it is not necessarily bounded away from zero. Barreira and Iommi \cite{bi1} proposed a definition of topological pressure when the roof function is bounded away from zero and established the variational principle in this case. This definition gave the pressure implicitly as the zero of a related function. Recently, Kempton \cite{ke} and independently Jaerisch, Kesseb\"ohmer and Lamei \cite{jkl}  gave a definition of pressure in the case that the roof function is not necessarily  bounded away from zero. This definition is analogous to that of the Gurevich pressure in the discrete case. The regularity of the pressure function for this type of flow (with roof function bounded away from zero) was studied by Iommi and Jordan \cite{ij}. Conditions in order for the pressure to be real analytic or to exhibit phase transitions were found.

We characterise potentials on the flow that have equilibrium measures  (Theorem \ref{thm:flow eqstate}). Moreover, we prove that when equilibrium measures do exist, they are unique (Theorem \ref{thm:uni}). We also extend to this continuous time setting the definitions of  positive recurrence, null recurrence and transience of a potential introduced in the discrete time by Sarig \cite{Sar99}.  There are several difficulties which must be addressed when trying to prove this type of result. To start with, the phase space is not compact, therefore the classical functional analytic approach can not be used directly. Moreover, there is no bijection between the set of invariant measures for the flow and the corresponding one for the base map. This prevents us from reducing the study of the thermodynamic formalism for the flow to that of the shift (this was the  strategy used by Bowen and Ruelle \cite{br} in the compact setting).

In order to give a relatively straightforward family of examples to which we can apply our theory, we define and study in some detail a semi-flow we call the \emph{renewal flow}. This is a suspension flow with base map the renewal shift (see Section~\ref{sec:reflow} for a precise definition).  This is a well-studied system which allows us to define a family of potentials exhibiting a range of thermodynamic behaviours.  For example, the base map can be set up to be very slowly mixing (like a model for the Manneville Pomeau \cite{ManPom} map), while the roof function can be designed so that the flow has one of a selection of thermodynamical properties.  Indeed in Examples~\ref{eg:better-worse} and \ref{eg:trans null} we demonstrate that we can arrange roof functions and potentials so that any pair from the set $\{$positive recurrent, null recurrent, transient$\}$  can be realised with the first behaviour on the base and the second on the corresponding flow.

Returning to the broad motivations for our work, the suspension flows here provide models for various non-uniformly hyperbolic flows where the thermodynamic formalism is not well developed.  These are systems which behave like Axiom A systems in most of the phase space, but not in all of it. It is possible for these systems to exhibit pathological behaviour in small parts of the domain. Interest in these systems is partially due to the novel dynamical features they exhibit (see for example the statistical laws and rates of decay of correlation that can occur in \cite{BalMel08, fmt,mt,m1,m2}).  Also, these systems have great importance in the program aimed at obtaining a global description of the space of dynamical systems  (see \cite{bdv}). While these systems still preserve some of the good properties of Axiom A (uniformly hyperbolic) systems, this is not enough to retain their  regular dynamical properties. Suspension flows over countable Markov shifts serve as symbolic models for some of these flows.  For example, Bufetov and Gurevich \cite{bg} and Hamenst\"adt \cite{h}  have coded Teichm\"uller flows in this way and have used this symbolic representation to prove the uniqueness of the measure of maximal entropy.  Another classical example of a flow that is modelled by this type of suspension flows is the geodesic flow over the modular surface (see \cite{ar}).  A further example which can be studied with the techniques we develop here is the geodesic flow defined over an hyperbolic manifold $\mathcal{H}^{N+1} / \Gamma$, where $\mathcal{H}^{N+1}$ is the hyperbolic space and $\Gamma$ a Schottky group. The relevant feature is that $\Gamma$ can have parabolic elements. In \cite{dp} Dal'bo and Peign\'e, generalising the Bowen-Series map to this setting, constructed a countable Markov partition and used it to code the geodesic flow with a suspension flow over a countable Markov partition. We stress, as they point out,  that it would also be possible to code this geodesic flow as suspension flow over a map with indifferent fixed points, very much in the spirit of what we study in sub-section \ref{ssec:mpflows}.  The coding constructed in \cite{dp} has been used, for example,  to study the  influence of cusps  on the homological
behaviour of closed geodesics \cite{bp1}. We would also like to point out that recent results by Sarig \cite{sa5} suggest that the techniques developed here could be applied to  a wide range of  non-uniformly hyperbolic flows obtained as suspensions of positive entropy surface diffeomorphisms.  Finally note that when coding a flow there is certain freedom in the choice of cross-section. In certain cases, choosing a different cross-section would yield a symbolic representation of the flow by a suspension flow over a countable Markov shift with a roof function bounded away from zero.  Since this could alter the base dynamics in non-trivial ways, in this paper we consider a fixed cross-section and work from there.

\subsection{Layout of the paper}
In Section~\ref{sec:prelim} we give the necessary definitions and results from the setting of countable Markov shifts; we also introduce suspension flows over countable Markov shifts and the notion of topological entropy for these flows. Section~\ref{sec:eq exist} begins with the definition of pressure for these flows, which has been introduced in \cite{sav,bi1,ke,jkl}. We then state and prove our first results: Theorem \ref{thm:flow eqstate} which characterises the potentials for which there exists an equilibrium measure, and Theorem~\ref{thm:uni} which gives uniqueness. Section~\ref{sec:ind} looks at inducing to a full shift and how the pressure for the induced potential can be related to the pressure for the original potential. In Section~\ref{sec:rec flow} we define the notions of recurrence and transience for suspension flows and relate these notions to the thermodynamics formalism for the shift map, proving a Ruelle-Perron-Frobenius-type result, Theorem~\ref{thm:flow rec}.   In the latter sections we focus on specific examples of suspension flows,
starting with the case of the renewal flow  in Section~\ref{sec:reflow}: in particular we consider the existence of equilibrium measures (including the existence of measures of maximal entropy) and phase transitions for the pressure function.  In this long section we highlight Examples~\ref{eg:better-worse} and \ref{eg:trans null} which give a complete picture of the possible pairs of behaviours for the base map and the flow.  Finally in Section~\ref{sec:MP flow} we apply our results to the setting where the base map $f$ is the non-uniformly expanding Manneville-Pomeau map and the roof function is $\log|f'|$.

\section{Preliminaries}
\label{sec:prelim}
In this section we collect all the definitions and results for countable Markov shifts and for suspension flows over countable Markov shifts that will be used in the latter sections.

\subsection{Recurrence, Transience and Thermodynamic formalism for countable Markov shifts.} \label{subsec:CMS}

Here we recall some results, mostly due to Sarig, that will be of use in the following sections. We note that Mauldin and Urba\'nski also developed a theory in this context \cite{mu1, mu, mubook}. However, the combinatorial restrictions they impose on the shifts are too strong for what we will require here.

Let $B$ be a transition matrix defined on the alphabet of natural numbers with zero. That is, the entries of the matrix
$B=B(i,j)_{\N_0  \times \N_0}$  are zeros and ones (with no row and no column
made entirely of zeros). The countable Markov shift $(\Sigma, \sigma)$
is the set
\[ \Sigma := \left\{ (x_n)_{n \in \N_0} : B(x_n, x_{n+1})=1  \text{ for every } n \in \N_0 \right\}, \]
together with the shift map $\sigma: \Sigma  \to \Sigma $ defined by
$\sigma(x_0, x_1, \dots)=(x_1, x_2,\dots)$.
We will always assume that this system is \emph{topologically mixing}, i.e., for each pair $x,y\in \N_0$, there exists $N \in \N$ such that for every $n > N$ there is a word $(i_0, \ldots, i_{n-1})\in \N_0^n$ such that $i_0=x, i_{n-1}=y$ and $B(i_k, i_{k+1})=1$ for all $0\le k\le n-2$.

Let $C_{i_0 \cdots i_{n-1}}$ denote the \emph{$n$-cylinder} consisting of all sequences $(x_0, x_1, \ldots)$ where $x_k=i_k$ for $0\le k\le n-1$.
The \emph{$n$-th variation} of $\phi:\Sigma \to \R$ is defined by
\[ V_{n}(\phi):= \sup \{| \phi(x)-  \phi(y)| : x,y \in \Sigma, x_{i}=y_{i}, 0 \leq i \leq n-1 \}. \]
We say that $ \phi$ is of \emph{summable variations} if $\sum_{n=1}^{\infty} V_n(\phi)< \infty$. We say that it is  \emph{locally H\"older} (with parameter $\theta$) if there exists $\theta \in (0,1)$ such that for all $n \geq 1$ we have  $V_{n}( \phi) \leq O( \theta^{n}). $
The first return time to $C_{i}$ is defined by
\begin{equation*} \label{eq:firstreturn}
 r_{i}(x) := \mathbbm{1}_{C_{i}}(x) \inf \{ n \geq 1 : \sigma^{n}x \in C_{i} \},
 \end{equation*}
 where $\mathbbm{1}_{C_{i}}$ is the indicator function of the cylinder $C_{i}$. Let
\[ X^i_n:= \{ x \in \Sigma : r_{i}(x)=n \}.\]
When it is clear what $i$ is, we will often drop the superscript.
Given $\phi$ a potential of summable variations and a 1-cylinder $C_i$, we define the partition functions
$$Z_n(\phi, C_i):=   \sum_{\sigma^{n}x=x}  \exp \left(\sum_{k=0}^{n-1} \phi(\sigma^{k}x)\right) \mathbbm{1}_{C_{i}}(x),$$
and
$$Z_n^*(\phi, C_i):= \sum_{\sigma^{n}x=x} \exp \left(\sum_{k=0}^{n-1}  \phi(\sigma^{k}x)\right) \mathbbm{1}_{X^i_n}(x). $$

The \emph{Gurevich Pressure} of $\phi$ was introduced by Sarig in \cite{Sar99}, generalising previous results by Gurevich \cite{gu1, gu2}.  It is defined by
\[ P_{\sigma}(\phi) = \lim_{n \rightarrow \infty} \frac{1}{n} \log Z_n(\phi, C_{i_0}).  \]
The limit always exists and if the system is topologically mixing then its value does not depend on the cylinder $C_{i_{0}}$ considered. The Gurevich pressure  is convex and  if
$\mathcal{K}:= \{ K \subset \Sigma : K \textrm{ compact and } \sigma\textrm{-invariant}, K \neq \emptyset \}$ then
\begin{equation*} \label{eq:press appr}
   P_{\sigma}(\phi) = \sup \{ P(\phi|K) : K \in \mathcal{K} \},
\end{equation*}
where $P(\phi| K)$ is the topological pressure of $\phi$ restricted to the compact set $K$ (for definition and properties see \cite[Chapter 9]{wa}).
Moreover, this notion of pressure satisfies the variational principle (see \cite{Sar99}):

\begin{teo}
Let $(\Sigma, \sigma)$ be a countable Markov shift and $\phi \colon \Sigma \to \R$ be a function of summable variations such that $\sup \phi < \infty$,  then
\[P_\sigma(\phi)= \sup \left\{ h_\sigma(\nu) + \int \phi~{\rm d} \nu : \nu \in \M_{\sigma} \text{ and } - \int \phi~{\rm d} \nu < \infty \right\},\]
where $ \M_{\sigma} $ denotes the set of $\sigma$-invariant probability measures and $h_\sigma(\nu)$ denotes the entropy of the measure $\nu$ (for a precise definition see  \cite[Chapter 4]{wa}).
\label{thm:CMS VP}
\end{teo}

The quantity $h_\sigma(\nu)+\int\phi~{\rm d}\nu$ is sometimes called the \emph{free energy} w.r.t. $\phi$, see \cite{Kel98} (although from a more physical perspective, the closer analogy is to minus the free energy).
A measure $\nu \in \M_{\sigma}$ attaining the infimum of the free energies, that is
\begin{equation}
P_\sigma(\phi)= h_\sigma(\nu) + \int\phi~{\rm d} \nu \label{eq:eq meas},
\end{equation}
 is called an \emph{equilibrium measure} for $\phi$. Buzzi and Sarig \cite{busa} proved that a potential of summable variations which is bounded above has at most one equilibrium measure.

Under certain combinatorial assumptions  on the shift (for example if $(\Sigma, \sigma)$ is a full-shift) the equilibrium measure also satisfies the Gibbs property.  We define this property here for general measures (i.e., not necessarily equilibrium measures).

\begin{defi}
We say that $\mu$ is a \emph{Gibbs measure} for $\phi$ if there exist $K,P\in \R$ such that for every $n\ge 1$, given an $n$-cylinder $C_{i_0 \cdots i_{n-1}}$,
$$\frac1K\le \frac{\mu(C_{i_0 \cdots i_{n-1}})}{e^{S_n\phi(x)-nP}}\le K$$
for any $x\in C_{i_0 \cdots i_{n-1}}$.
\end{defi}

Note that we will usually have $P=P_\sigma(\phi)$, for example in the full shift example mentioned above.

Potentials  can be classified according to their recurrence properties as follows (see \cite{Sar99,Sar01}); note that by topological mixing and summable variations, these definitions are independent of the choice of cylinder $C_i$.

\begin{defi} \label{def:rec pots}
Let $\phi$ be a potential of summable variations with finite Gurevich pressure $P_\sigma(\phi)= \log \lambda$. We say that $\phi$ is
\begin{itemize}
\item \emph{recurrent} if
\[ \sum_{n \geq 1} \lambda^{-n} Z_n(\phi, C_i)= \infty,\]
\item
and \emph{transient} otherwise.
\end{itemize}
Moreover, we  say that $\phi$ is
\begin{itemize}
\item[-] \emph{positive recurrent} if it is recurrent and
\[ \sum_{n \geq 1} n\lambda^{-n} Z_n^*(\phi, C_i) <\infty;\]
\item[-] \emph{null recurrent} if it is recurrent and
\[ \sum_{n \geq 1} n\lambda^{-n} Z_n^*(\phi, C_i) =\infty.\]
\end{itemize}

\end{defi}

Consider the Ruelle operator defined formally in some space of functions by:
\begin{equation*}
L_{\phi}g(x):= \sum_{\sigma y=x} \exp( \phi (y)) g(y).
\end{equation*}
Note that if $\|L_{\phi}1 \|_{\infty} < \infty$ then the transfer operator is a bounded operator on the Banach space $C_B:=\left\{g \in C(\Sigma) : \|g\|_{\infty}< \infty  \right\}$, where $C(\Sigma)$ is the space of continuous functions $g:\Sigma \to\R$ (see \cite{Sar99,Sar01a}).

A measure $m$ on $\Sigma$ is called \emph{$\phi$-conformal} if $m(\sigma(A))=\int_A e^{-\phi}\rm{d}m$ whenever $A$ is measurable and $\sigma$ is injective on $A$.  Moreover, a measure $\nu$ on $\Sigma$ is called \emph{conservative} if any measurable set $W\subset \Sigma$ such that  the sets $\{f^{-n}(W)\}_{n=0}^{\infty}$ are disjoint has $\nu(W)=0$.

Sarig \cite{Sar01} generalises the Ruelle Perron Frobenius Theorem to countable Markov shifts.
\begin{teo}[RPF Theorem] \label{thm:RPF}
Let $(\Sigma, \sigma)$ be a countable Markov shift and $\phi$ a potential of summable variations  of finite Gurevich pressure $\log \lambda$.   If $\phi$ is
\begin{enumerate}[label=({\alph*}),  itemsep=0.0mm, topsep=0.0mm, leftmargin=7mm]
\item positive recurrent then there exists a conservative $\phi$-conformal measure $m$ and a continuous function $h$ such that $L_{\phi}^{*} m= \lambda m$,
$L_{ \phi} h= \lambda h$ and  $ h~{\rm d} m < \infty$;
\item null recurrent then there exists a conservative $\phi$-conformal measure $m$ and a continuous function $h$ such that  $L_{\phi}^{*} m= \lambda m$,
$L_{\phi} h= \lambda h$ and $\int h~{\rm d} m = \infty$;
\item transient then there is no conservative $\phi$-conformal measure.
\end{enumerate}
\end{teo}

In the recurrent case, we call the measure $\nu =h {\rm d} m$ the \emph{Ruelle-Perron-Frobenius (RPF)} measure of $\phi$. In the case where $\phi$ is positive recurrent, the total measure of the space is finite, and we rescale to make $\nu$ a probability measure.  This is an equilibrium measure for $\phi$ provided $-\int  \phi~{\rm d} \nu< \infty$.
Daon proved in \cite[Proposition 4.2]{da} that the RPF-measure is unique.

\begin{prop} \label{prop:uniq-rpf}
Let $(\Sigma, \sigma)$ be a topologically mixing countable Markov shift and $\phi$ a potential of summable variations then there is at most one RPF-measure for $\phi$.
\end{prop}

The first parts of the following result  were  established in  \cite[Theorem 2]{Sar01} and provides a version of the variational principle for infinite invariant measures. The final part of the second statement is pointed out in \cite[Remark 4]{Sar01}.

\begin{teo}\label{thm:infiniteRPF}
Let $(\Sigma, \sigma)$ be a countable Markov shift and $\phi \colon \Sigma \to \R$ be a recurrent locally H\"older potential of   finite Gurevich pressure.
\begin{enumerate}[label=({\alph*}),  itemsep=0.0mm, topsep=0.0mm, leftmargin=7mm]
\item For every conservative ergodic invariant measure $\nu$ which is finite on cylinders, if $\int (P_\sigma(\phi) -\phi) ~{\rm d}\nu<\infty$ then $h_\sigma(\nu) \leq  \int (P_\sigma(\phi) -\phi) ~{\rm d}\nu$.
\item Let $h$ and $m$ be the density and the conformal measure provided by the Ruelle Perron Frobenius Theorem and $\nu=hm$. If $ \int (P_\sigma(\phi) -\phi)~{\rm d}\nu< \infty$ then
$h_\sigma(\nu) =  \int (P_\sigma(\phi) -\phi) ~{\rm d}\nu$. Moreover, $\nu$ is the unique measure (up to a constant) such that $h_\sigma(\nu) =  \int (P_\sigma(\phi) -\phi) ~{\rm d}\nu$ and $\nu$ has a generator of finite entropy.
\end{enumerate}
\end{teo}

\begin{rem} \label{rem:rpf-uni}
Observe that in the case that the RPF measure is finite, the final part of the second statement above is a more standard statement of uniqueness of equilibrium state for $\phi$.  The subtlety here is that it also gives a form of uniqueness for infinite measures.
\end{rem}

\subsection{Inducing and the variational principle}
\label{ssec:induce discr}
In this subsection we study the inducing procedure in the context of countable Markov shifts as well as to prove the variational principle without the $\sup \phi < \infty$ assumption considered in Theorem \ref{thm:CMS VP}.

The combinatorial structure of a countable Markov shift has several important consequences in the properties of its corresponding thermodynamic formalism. For example, if $(\Sigma, \sigma)$ is a full-shift then locally H\"older potentials $\phi$ of finite entropy have corresponding Gibbs measures and the pressure function $t \mapsto P(t \phi)$ (when finite) is real analytic (see \cite{Sar01a}). Moreover, if $\Sigma$ does not satisfy a certain combinatorial assumption (the so called BIP property, see \cite{Sar03} for a precise definition) then locally H\"older potentials do not have corresponding Gibbs measures \cite[Theorem 1]{Sar03}. The inducing procedure in the context of topologically mixing countable Markov shifts consists of associating to any system $(\Sigma, \sigma)$ a full-shift on a countable alphabet. The idea being to solve problems in this new (better behaved) system and then to translate them back into the original system.

Let $(\Sigma, \sigma)$ be a topologically mixing countable Markov shift.  Fix a symbol in the alphabet, say $a \in \N$. The \emph{induced system} over the cylinder $C_a$, denoted by $(\overline{\Sigma}, \overline{\sigma})$,  is the full-shift defined on the alphabet
\begin{equation*}
\left\{ C_{ai_1 \dots i_m} : i_j \neq a \text{ and }  C_{ai_1 \dots i_ma} \neq \emptyset \right\}.
\end{equation*}
As defined in Section \ref{subsec:CMS} let $r_a(x)$ be  the \emph{first return time} to the cylinder $C_a$. For every
potential $\phi:\Sigma \to \R$ we define the induced potential by
\begin{equation}
\overline{\phi}:=\left( \sum_{k=0}^{r_a(x)-1} \phi \circ \sigma^k \right).
\label{eq:overline phi}
\end{equation}
Note that  if $\phi$ is  locally H\"older then so is $\overline{\phi}$ and that if $\phi$ has summable variations then $V_1(\overline{\phi})<\infty$ and $\lim_{n\to\infty}V_n(\phi)=0$.

There is a relation between invariant measures for the induced system $\overline{\mu}$ and measures on the original system. Indeed, an invariant probability  measure $\overline{\mu}$ on the induced systems such that $\int r_a ~{\rm d} \overline{\mu} < \infty$ can be projected onto an invariant probability measure $\mu$ on the original system in the following way
\begin{equation} \label{eq:proj meas}
\mu(A)=\frac1{\int r_a~{\rm d}\overline\mu}\sum_{n=1}^{\infty}\sum_{k=0}^{n-1}\overline\mu(\sigma^{-k}(A)\cap X_n^a),
\end{equation}
where $X_n^a:= \{x \in \Sigma : r_a(x)=n \}$. We say that $\overline{\mu}$ is the \emph{lift} of $\mu$ and that $\mu$ is the \emph{projection} of $\overline{\mu}$.
We can relate the integral of a potential with respect to a measure and the integral of the induced potential with respect to the lifted measure. We also have that the entropy of a measure is related to that of the lifted measure.

\begin{rem} \label{rmk:Kac and Ab}
Let $\nu \in \M_{\sigma}$ be ergodic and satisfy that $\nu(C_a)>0$. Denote by $\overline{\nu} \in \M_{\overline{\sigma}}$ its lift to the induced system. Let $\phi:\Sigma \to \R$ be a potential of summable variations and let $\overline{\phi}$ be the corresponding induced potential. If $\int r_a ~{\rm d} \overline{\nu} < \infty$ then
\[ \int \phi~{\rm d} \nu = \frac{\int \overline{\phi} ~{\rm d} \overline{\nu}}{\int r_a ~{\rm d} \overline{\nu}}. \]
This is known as Kac's Formula/Lemma. Similarly for entropy, Abramov's Formula (see \cite{ab2}) gives an analogous statement for entropy:
\[ h_\sigma(\nu) = \frac{h_{\overline\sigma}(\overline{\nu})}{\int r_a ~{\rm d} \overline{\nu}}. \]
\end{rem}

Making use of the inducing scheme we will prove the variational principle without the assumption of the potential being bounded above. We use the following simple lemma in the proof.

\begin{lema}\label{lem:finite}
Let $(\Sigma_f, \sigma)$ be the full-shift on a countable alphabet and $\overline{\phi}:\Sigma_f \to \R$ a  potential of finite pressure such that $V_1(\overline{\phi})<\infty$ and $\lim_{n\to\infty}V_n(\overline{\phi})=0$ then $\sup \overline{\phi} < \infty$.
\end{lema}

\begin{proof}
We prove the contrapositive by a similar argument to  \cite[Proposition 2.1.9]{mubook}. If $\sup\overline{\phi}= \infty$ then, since $V_1(\overline{\phi})<\infty$, we have that $\sup_{a\in\N}\inf_{x\in C_a}\overline{\phi}(x)=\infty$. Thus if we let $\delta_a$ be the invariant Dirac measure supported at the fixed point $(a,a,\ldots)$ then by the variational principle for the full shift under the condition that $V_n(\overline\phi)\to 0$ as $n\to \infty$ (see \cite[Theorem 2.1.8]{mubook}), we have that
$$P(\overline{\phi})\geq \sup_{a\in\N}\left\{\int \overline{\phi}~{\rm d}\delta_a\right\}=\infty$$
which completes the proof.
\end{proof}
We can now extend the variational principle to the case when we do not assume $\phi$ is bounded above.
\begin{teo} \label{thm:var-non-sup}
Let $(\Sigma,\sigma)$ be a topologically mixing countable Markov shift. Let $\phi:\Sigma \to \R$ be a potential of summable variations then
\begin{equation*}
P(\phi)= \sup\left\{h_\sigma(\mu) + \int \phi ~{\rm d}\mu : \mu \in \M \text{ and } -\int \min\{\phi,0\} ~{\rm d} \mu < \infty		\right\}.
\end{equation*}
\end{teo}

\begin{proof}
Note that the approximation by compact invariant subsets property of the Gurevich pressure and Walters variational principle for continuous maps on compact sets \cite{W} implies that for every $\phi$ of summable variations we have
\[P(\phi) \leq  \sup\left\{h_\sigma(\mu) + \int \phi ~{\rm d}\mu : \mu \in \M \text{ and } -\int \min\{\phi,0\} ~{\rm d} \mu < \infty		\right\}.\]
If $P(\phi)= \infty$ this completes the proof.

In order to prove the other inequality we assume that $P(\phi) <\infty$ and $\phi$ has summable variations. Assume further  that $\phi$ has zero pressure $P(\phi)=0$ (otherwise consider the potential $\phi -P(\phi)$).
The induced potential $\overline{\phi}$ on the cylinder $C_a$ is such that $P(\overline{\phi}) \leq 0$ (this follows as in the proof of Lemma 4.1 of \cite{IomTod10}, with the main difference being that since $\overline\phi$ is not assumed to have summable variations, we use the corresponding theory of Chapter 2 of [MU3] in place of the theory of Sarig used there).
If $\mu \in \M$ then there exists a cylinder $C_{a'}$ such that $\mu$ lifts to a measure $\overline{\mu}$ on the induced system as in \eqref{eq:proj meas}. Moreover, $\int \phi ~{\rm d} \mu > -\infty$ if and only if
$\int \overline{\phi} ~{\rm d} \overline{\mu} > -\infty$. Recall that the induced system is a full shift and that $\overline{\phi}$ satisfies the assumptions for Lemma~\ref{lem:finite}. Therefore, $\sup \overline{\phi} <\infty$ and the variational principle holds (see \cite[Theorem 2.1.8]{mubook}). Since the induced pressure is non-positive we have that
\[h_{\overline\sigma}(\overline{\mu}) + \int \overline{\phi}  ~{\rm d} \overline{\mu} \leq P(\overline{\phi}) \leq 0.\]
The return time is integrable so Abramov and Kac's Formulas (see Remark~\ref{rmk:Kac and Ab}) imply
\[h_\sigma(\mu) + \int \phi ~{\rm d} \mu \leq 0=P(\phi).\]
This holds for every invariant measure, so
\[ \sup \left\{h_\sigma(\mu) + \int \phi ~{\rm d}\mu : \mu \in \M \text{ and } -\int \min\{\phi,0\} ~{\rm d} \mu < \infty		\right\} \leq P(\phi). \]
\end{proof}

\subsection{Suspension semi-flows}
\label{ssec:semi-flow}
Let $(\Sigma, \sigma)$ be a countable Markov shift and $\tau \colon \Sigma \to \R^+$ be a positive continuous function such that for every  $x\in\Sigma$ we have
\begin{equation} \label{eq:Hopf cond}
\sum_{i=0}^{\infty}\tau(\sigma^i x)=\infty.
\end{equation}
We need this assumption to assure that the semi-flow is well defined for every time $t>0$.
Consider the space
\begin{equation*}\label{eq:flow phase }
Y= \{ (x,t)\in \Sigma  \times \R \colon 0 \le t \le\tau(x)\}/\sim
\end{equation*}
where $(x,\tau(x))\sim (\sigma(x),0)$ for
each $x\in \Sigma $.

The \emph{suspension semi-flow} over $\sigma$
with \emph{roof function} $\tau$ is the semi-flow $\Phi = (
\varphi_t)_{t \ge 0}$ on $Y$ defined by
\[
 \varphi_t(x,s)= (x,
s+t) \ \text{whenever $s+t\in[0,\tau(x)]$.}
\]
In particular,
\[
 \varphi_{\tau(x)}(x,0)= (\sigma(x),0).
\]
In the case of two-sided Markov shifts we can define a suspension
flow $(\varphi_t)_{t\in \R}$ in a similar manner.

\subsection{Invariant measures} \label{ssec:meas AK}
In this section we discuss the relation between invariant measures for the flow and invariant measures for the base map.
\begin{defi}
A probability measure $\mu$  on $Y$  is
\emph{$\Phi$-invariant} if $\mu(\varphi_t^{-1}A)= \mu(A)$ for every
$t \ge 0$ and every measurable set $A \subset Y$. Denote by $\M_\Phi$ the space of $\Phi$-invariant probability
measures on $Y$
\end{defi}

The space $\M_\Phi$ is closely related to the space $\M_\sigma$ of $\sigma$-invariant probability measures on $\Sigma $. Let us consider the space of $\sigma$-invariant measures for which $\tau$ is integrable,
\begin{equation}
\M_\sigma(\tau):= \left\{ \mu \in \mathcal{M}_{\sigma}: \int \tau ~{\rm d} \mu < \infty \right\}.
\end{equation}
Let $m$ denote  one-dimensional Lebesgue measure.  As the flow direction is one-dimensional and $m$ is the unique measure which is invariant under all translations, it follows that a $\Phi$-invariant probability measure will be of the form $C\mu\times m|_{Y}$ for $\mu \in \M_\sigma(\tau)$ and some $C>0$.  Indeed, it follows directly from classical results by Ambrose and Kakutani \cite{ak} that \[\frac{(\mu \times m)|_{Y} }{(\mu \times m)(Y)} \in \M_{\Phi}.\]
This suggests that the study of the map $R \colon \M_\sigma \to \M_\Phi$, defined by
\begin{equation*} \label{eq:R map}
R(\mu)=\frac{(\mu \times m)|_{Y} }{(\mu \times m)(Y)}
\end{equation*}
should allow for the translation of problems from the flow onto the shift map. This is indeed the case:
\begin{enumerate}
\item When $(\Sigma , \sigma)$ is a sub-shift of finite type defined over a finite alphabet, a compact case, the map $R$ is a bijection.

\item If $(\Sigma , \sigma)$ is a countable Markov shift and $\tau:\Sigma \to \R$ is  not  bounded above then it is possible for  there to be a measure $\nu \in \M_\sigma\sm \M_\sigma(\tau)$, i.e., such that $\int \tau ~{\rm d} \nu  = \infty$. In this situation the measure $R(\nu)$ is an infinite invariant measure for $\Phi$. Hence, the map $R(\cdot)$ is not well defined. Nevertheless, it follows directly from the results by Ambrose and Kakutani \cite{ak} that if $\tau$ is uniformly bounded away from zero then
 the map $R:\M_\sigma(\tau) \to  \M_{\Phi}$ is bijective.

\item   If $(\Sigma , \sigma)$ is a countable Markov shift and $\tau:\Sigma \to \R$ is  not  bounded away from zero then it is possible (see Section~\ref{sec:Hof roof}) that for an infinite (sigma-finite) $\sigma$-invariant measure $\nu$  we have $\int  \tau ~{\rm d} \nu <\infty$. In this case the measure $(\nu \times m)|_{Y} /(\nu \times m)(Y) \in \M_\Phi$. In such a situation,   the map $R$ is not surjective.
 \end{enumerate}

The situation in the first two cases is usually simpler since every measure in $ \M_\Phi$ can be written as $R(\nu)$, where $\nu \in \M_{\sigma}$.  Therefore, the ergodic properties of the flow can be reduced to the ergodic properties of the base. If the roof function is not bounded away from zero this is no longer the case.  However:

\begin{rem} \label{rmk:conser}
The time assumption given by equation \eqref{eq:Hopf cond}  implies, by a result of Hopf (see \cite[Proposition 1.1.6]{aa} and \cite{hop}), that every sigma-finite measure $\nu\in \M_\sigma$ defined on the base such that $\int \tau ~{\rm d} \nu < \infty$ is conservative.
 \end{rem}

Given a  function $g \colon Y\to\R$ we define the function
$\Delta_g\colon\Sigma\to\R$~by
\[
\Delta_g(x)=\int_{0}^{\tau(x)} g(x,t) \, ~{\rm d}t,
\]
whenever for every $x \in \Sigma$ we have that  $g(x,t)$ is integrable on $[0, \tau(x)]$ with respect to $t$ (see Remark \ref{rmk:exten} for regularity conditions).  If $\mu \in \M_{\Phi}$ is the normalisation of  $R(\nu)$ (note $\nu$ can be an infinite measure, as long as  $\int \tau ~{\rm d}\nu < \infty)$ then
\begin{equation*} \label{eq:rela}
\int_{Y} g \, ~{\rm d}R(\nu)= \frac{\int_\Sigma \Delta_g\, ~{\rm d}
\nu}{\int_\Sigma\tau \, ~{\rm d} \nu}.
\end{equation*}

\begin{rem}[Extension of potentials defined on the base] \label{rmk:exten}
Let $\phi \colon \Sigma  \to \R$ be a locally H\"older potential. It is shown in
\cite{brw} that there exists a function $g \colon Y \to
\R$ which is continuous in the Bowen-Walters metric such that $\Delta_g=\phi$.
\end{rem}

\subsection{Abramov's formula} \label{ssec:abra}
The entropy of a flow with respect to an invariant measure,  denoted $h_\Phi(\mu)$, can be defined as  the entropy of the corresponding time one map. The following  classical result obtained by
Abramov \cite{a} relates the entropy of a measure for the flow with the entropy of a measure for the base map: this is the continuous-time analogue of the Abramov formula in Remark~\ref{rmk:Kac and Ab}.
\begin{prop}[Abramov]
Let $\mu \in \M_{\Phi}$ be such that  $\mu=(\nu \times m)|_{Y} /(\nu \times m)(Y)$, where $\nu \in \M_{\sigma}$ then
\begin{equation}
h_{\Phi}(\mu)=\frac{h_{\sigma}(\nu)}{\int \tau ~{\rm d} \nu}.
\end{equation}
\label{prop:Abr}
\end{prop}

The result of Abramov holds for any suspension flow with positive (not necessarily bounded away from zero) roof function and for any invariant (not necessarily ergodic) finite measure for the flow $\mu$ that can be written as $\mu = R(\nu)$, where $\nu$ is an invariant probability measure for the base with $\int \tau ~{\rm d} \nu < \infty$. This settles  the case when the roof function is bounded away from zero, since every invariant measure for the flow is of that form. When the roof function is not bounded away from zero there are invariant measures for the flow $\mu$ that are not of  that form. But instead,
 $\mu = R(\nu)$, where $\nu$ is an infinite invariant measure for the shift. Savchenko \cite[Theorem 1]{sav} proved that if $\mu$ is  ergodic then  Abramov's formula still holds where Krengel's definition for entropy of an conservative infinite measure is used. Let $\mathcal{E}_\Phi$ be  the set of ergodic $\Phi$-invariant measures and recall that $\int \tau ~{\rm d}\nu < \infty$ implies that $\nu$ is conservative.
 \begin{prop}[Savchenko]
Let $\Phi$ be a suspension semi-flow defined over a countable Markov shift with positive roof function $\tau$. Let $\mu \in \mathcal{E}_\Phi$ be  such that  $\mu=(\nu \times m)|_{Y} /(\nu \times m)(Y)$, where $\nu$ is a sigma-finite (infinite) invariant measure for the shift with
 $\int \tau ~{\rm d}\nu < \infty$. Then
\begin{equation}
h_{\Phi}(\mu)=\frac{h_{\sigma}(\nu)}{\int \tau  ~{\rm d} \nu}.
\end{equation}
\label{prop:Sav}
\end{prop}

\begin{coro}
Let $\mu \in \M_{\Phi}$ be such that $\mu=(\nu \times m)|_{Y} /(\nu \times m)(Y)$ for $\nu\in \M_\sigma$.  Then  $h_{\Phi}(\mu)= \infty$ if and only if
$h_{\sigma}(\nu)= \infty$.
\label{cor:ent inf}
\end{coro}

When the phase space is non-compact  there are several different notions of topological entropy of a flow, we will consider the following,
\begin{defi}
The \emph{topological entropy} of the suspension flow $(Y ,\Phi)$ denoted by $h(\Phi)$ is defined by
\begin{equation*}
h(\Phi):= \sup \left\{  h_{\Phi}(\mu) : \mu \in \mathcal{E}_\Phi\right\},
\end{equation*}
where $\mathcal{E}_\Phi $ is the set of ergodic $\Phi$-invariant measures.
\end{defi}
A measure $\mu \in \mathcal{E}_\Phi$ such that $h(\Phi)= h_{\Phi}(\mu)$ is called a \emph{measure of maximal entropy}. Since the phase space is not compact, there exist suspension flows of finite entropy with no measure of maximal entropy (see Example \ref{eg:Hof nlogn}). Moreover, there are suspension flows for which the measure of maximal entropy $\mu$ is of the form $\mu = R(\nu)$, where $\nu$ is an infinite invariant measure for the shift map (see Example \ref{lem:MP}). In Corollary \ref{cor:meas_max}  we establish criteria to determine when suspension flows have (or do not have) measures of maximal entropy.

\subsection{Flows and semi-flows} \label{ssec:invert flow}

Sinai remarked that it is possible to translate problems regarding thermodynamic formalism from flows to semi-flows.  Indeed, denote by  $(\Sigma^*, \sigma)$  a two-sided countable Markov shift.
Two continuous functions $\phi, \gamma \in C(\Sigma^*)$ are said to be \emph{cohomologous} if there exists a bounded continuous function $\psi \in C(\Sigma^*)$ such
that $\phi= \gamma +\psi \circ \sigma -\psi$. The pressure function is invariant under cohomology and so are the thermodynamic properties, such as the existence of equilibrium measures.  The following result is  proved in the same way as in \cite[Section 3]{Sin72}, \cite[Lemma 1.6]{Bow08} or  \cite[Theorem 2]{cq}, where we note that we require the usual two-sided analogue of variation here.   We remark that a corresponding result  on Walters potentials was recently proved by Daon, \cite[Theorem 7.1]{da}.

\begin{prop} \label{prop:two-sided}
If $\phi \in C(\Sigma^*)$ has summable variation, then there
exists $\gamma \in C(\Sigma^*)$ of summable variation cohomologous to $\phi$ such that
$\gamma(x)=\gamma(y)$ whenever $x_i=y_i$ for all $i \geq 0$ (that is,
$\gamma$ depends only on the future coordinates).  This result also holds in the class of H\"older potentials.
\end{prop}

Note that in the case of H\"older potentials, the H\"older exponent $\alpha$ is halved when going from two-sided setting to the one-sided.  Proposition~\ref{prop:two-sided} implies that thermodynamic formalism for suspension flows can be studied by considering suspension semi-flows.

\section{Existence and uniqueness of equilibrium measures}
\label{sec:eq exist}
The definition of pressure for suspension flows over countable Markov shifts has been given with different degrees of generality by Savchenko \cite{sav}, Barreira and Iommi \cite{bi1}, Kempton \cite{ke} and Jaerisch, Kesseb\"ohmer and  Lamei \cite{jkl}. These definitions can be summarised as follows

\begin{teo} \label{thm: flow pres}
Let $(\Sigma, \sigma)$ a topologically mixing countable Markov shift and $\tau:\Sigma \to \R$ a positive function of summable variations satisfying  \eqref{eq:Hopf cond}.  Let $(Y, \Phi)$ be the suspension semi-flow over $(\Sigma, \sigma)$ with roof function $\tau$. Let $g:Y \to \R$ be a function such that $\Delta_g:\Sigma \to \R$ is of summable variations. Then the following equalities hold
\begin{eqnarray*}
P_{\Phi}(g)&:=&\lim_{t \to \infty} \frac{1}{t} \log \left(\sum_{\phi_s(x,0)=(x,0), 0<s \leq t} \exp\left( \int_0^s g(\phi_k(x,0)) ~{\rm d}k \right) \mathbbm{1}_{C_{i_0}}(x) \right) \\
&=& \inf\{t \in \R : P_{\sigma} (\Delta_g - t \tau) \leq 0\} =\sup \{t \in \R : P_{\sigma} (\Delta_g - t \tau) \geq 0\} \\
&=& \sup \{ P_{\Phi|K}(g) : K\in \cK \},
\end{eqnarray*}
where $\cK$ is the set of all compact and $\Phi$-invariant sets and  $P_{\Phi|K}(g)$ is the
classical topological pressure of the potential $g$ restricted to the
compact and $\Phi$-invariant set $K$.
\end{teo}

The variational principle has been proved in the context of suspension flows defined over countable Markov shifts (analogous to Theorem~\ref{thm:CMS VP})  with different degrees of generality (see \cite{bi1, jkl, ke,sav}). The version we will be interested here is the following:

\begin{teo}[Variational Principle] Under the same assumptions of Theorem \ref{thm: flow pres} we have
\begin{equation*}
P_\Phi(g)=\sup \left\{ h_{\Phi}(\mu) +\int_Y g ~{\rm d} \mu : \mu\in
\mathcal{E}_{\Phi} \text{ and } -\int_Y g \, ~{\rm d}\mu <\infty \right\},
\end{equation*}
where $\mathcal{E}_\Phi $ is the set of ergodic $\Phi$-invariant measures.
\end{teo}

Note that the set of measures considered in the variational principle is that of ergodic flow-invariant measures and not the the set of flow-invariant probability  measures. If the roof function is bounded away from zero the variational principle holds in complete generality \cite{bi1}. However, at present time the available proofs in the case that the roof function is not bounded away from zero \cite{jkl, ke} only hold for ergodic measures. The reason for this being that Abramov formula only holds for these measures (see Subsection \ref{ssec:abra}). Another way to approach this issue could be by taking the time-1 map of the flow and applying the usual ergodic decomposition argument, but this should require extra assumptions, for example on the regularity of the potential $g$ and the roof $\tau$.  Also note that it follows from our definitions that $P_{\Phi}(0)=h(\Phi)$.

Similarly to \eqref{eq:eq meas}, we define:

\begin{defi}
A measure $\mu \in \mathcal{E}_{\Phi}$ is called an \emph{equilibrium measure} for $g$ if
\begin{equation*}
P_\Phi(g)= h_\Phi(\mu) + \int g ~{\rm d} \mu.
\end{equation*}
\end{defi}

The next theorem is our first main result. In this general context, with roof function not necessarily bounded away from zero, we establish conditions to ensure the existence of equilibrium measures. Note that we are not ready to state a complete analogue of Theorem~\ref{thm:RPF} for semi-flows since we haven't yet defined recurrence and transience in that case (see Section~\ref{sec:rec flow} for this).

\begin{teo}\label{thm:flow eqstate}
Let $\Phi$ be a finite entropy suspension semi-flow on $Y$ defined over a countable Markov
shift  $(\Sigma, \sigma)$  and a locally H\"older roof function $\tau$  satisfying  \eqref{eq:Hopf cond}. Let $g \colon Y \to \R$ be a continuous function such
that $\Delta_g$ is locally H\"older. In the following cases there exists an equilibrium measure for $g$;
\begin{enumerate}[label=({\alph*}),  itemsep=0.0mm, topsep=0.0mm, leftmargin=7mm]
\item If $P_\sigma(\Delta_g -P_{\Phi}(g) \tau)=0$ and $\Delta_g -P_{\Phi}(g) \tau$ is positive recurrent with equilibrium measure $\nu_g$ satisfying $\int \tau ~{\rm d} \nu_g < \infty$;
\item   If $P_\sigma(\Delta_g -P_{\Phi}(g) \tau)=0$ and the potential $\Delta_g -P_{\Phi}(g) \tau$ is null recurrent with infinite RPF measure $\nu_g$  and  $\int \tau ~{\rm d}\nu_g < \infty$.
   \end{enumerate}
In any other case the potential $g$ does not have an equilibrium measure. Thus there is no equilibrium measure for $g$ when
\begin{enumerate}[label=({\roman*}),  itemsep=0.0mm, topsep=0.0mm, leftmargin=9mm]
\item $P_\sigma(\Delta_g -P_{\Phi}(g) \tau)<0$;
 \item $P_\sigma(\Delta_g -P_{\Phi}(g) \tau)=0$ and the potential $\Delta_g -P_{\Phi}(g) \tau$ is positive recurrent with equilibrium measure $\nu_g$ such that  $\int \tau ~{\rm d} \nu_g = \infty$;
 \item $P_\sigma(\Delta_g -P_{\Phi}(g) \tau)=0$ and the potential $\Delta_g -P_{\Phi}(g) \tau$ is null recurrent with infinite RPF measure $\nu_g$  and $\int \tau ~{\rm d} \nu_g = \infty$;
\item  $\Delta_g -P_{\Phi}(g) \tau$ is transient.
 \end{enumerate}
\label{thm:flow VP}
 \end{teo}

Before proving this key theorem, we state and prove a uniqueness result.

\begin{teo} \label{thm:uni}
Let  $(Y, \Phi)$ be a suspension semi-flow defined over  a topologically mixing countable Markov shift $(\Sigma , \sigma)$ with a locally H\"older roof function $\tau$. Let $g:Y \to \R$ be such that $\Delta_g$ is locally H\"older. Then $g$ has at most one equilibrium state.
\end{teo}

\begin{proof}
If $g$ has an equilibrium state $\mu$ then $P_{\sigma}(\Delta_g -P(g)\tau)=0$,
the potential $\Delta_g -P(g)\tau$ is recurrent, and we can express $\mu$ as $\mu= (\nu \times m)_Y/((\nu \times m) (Y))$ for some $\sigma$-invariant $\nu$.  Theorem~\ref{thm:infiniteRPF} then implies that $\nu$ is unique up to rescaling.  Hence $\mu$ is the unique equilibrium state for $g$.
\end{proof}

Simplifying to the case of the potential which is constant zero, and the corresponding measures of maximal entropy, we have the following.

\begin{coro}[Measures of maximal entropy] \label{cor:meas_max}
Let $\Phi$ be a finite entropy suspension semi-flow on $Y$ defined over a countable Markov shift  $(\Sigma, \sigma)$   with locally H\"older  roof function $\tau$ satisfying  \eqref{eq:Hopf cond}.
\begin{enumerate}[label=({\alph*}),  itemsep=0.0mm, topsep=0.0mm, leftmargin=7mm]
\item If $P_\sigma(-h(\Phi) \tau)=0$ and $-h(\Phi) \tau$ is positive recurrent with equilibrium measure $\nu$ satisfying $\int \tau ~{\rm d} \nu < \infty$ then there exists  a unique measure of maximal entropy.
\item   If $P_\sigma(-h(\Phi) \tau)=0$ and the potential $h(\Phi) \tau$ is null recurrent with infinite RPF measure $\nu$  and  $\int \tau ~{\rm d} \nu < \infty$ then there exists a unique measure of maximal entropy.
   \end{enumerate}
In any other case the flow  does not have a measure of maximal entropy.
\end{coro}

To prove Theorem \ref{thm:flow eqstate}, we require two lemmas.

\begin{lema} \label{lem:neg base}
If $P_\sigma(\Delta_g -P_{\Phi}(g) \tau)<0$ then there are no equilibrium measures for $g$.\end{lema}

\begin{proof}
We will show that for any measure $\mu \in \M_{\Phi}$ we have
$h_\Phi(\mu) + \int g ~{\rm d} \mu<P_\Phi(g)$. Assume first that $\mu= R(\nu)$ where
$\nu \in \M_{\sigma}$,  i.e., $\nu$ is a probability measure. Since $P_\sigma(\Delta_g -P_{\Phi}(g) \tau)<0$,
Theorem~\ref{thm:flow VP} implies
\begin{eqnarray*}
h_\sigma(\nu) + \int \Delta_g ~{\rm d}\nu -  P_{\Phi}(g) \int \tau ~{\rm d} \nu <0,
\end{eqnarray*}
thus
\begin{eqnarray*}
\frac{h_\sigma(\nu)}{\int \tau ~{\rm d} \nu}  + \frac{\int \Delta_g ~{\rm d} \nu}{{\int \tau ~{\rm d} \nu}} =
h_\Phi(\mu) + \int g ~{\rm d}\mu < P_{\Phi}(g).
\end{eqnarray*}
Therefore, no measure $\mu \in \M_{\Phi}$ of the form $\mu= R(\nu)$, where
$\nu \in \M_{\sigma}$, can be an equilibrium measure for $g$.

Let us assume now that
 $\mu= R(\nu)$ where $\nu$ is an infinite invariant measure such that $\int \tau~{\rm d}Ê\nu < \infty$.  Note that since the flow has finite entropy, $h_\Phi(\mu)=h_\Phi(R(\nu)) < \infty$, so Abramov's formula implies that $h_\sigma(\nu) < \infty$.
 Assume by way of contradiction that the measure $\mu$ is an equilibrium measure for $g$. In particular, since $h_\Phi(\mu)< \infty$, this implies that $\int g ~{\rm d} \mu <\infty$. Since  $\tau \in L^1(\nu)$  we have $\Delta_g \in L^1(\nu)$ and
 \begin{equation*}
 P_{\Phi}(g)= h_\Phi(\mu) + \int g ~{\rm d} \mu= \frac{h_\sigma(\nu)}{\int \tau ~{\rm d} \nu} + \frac{\int \Delta_g ~{\rm d}\nu}{\int \tau ~{\rm d} \nu}.
  \end{equation*}
This implies,
\begin{equation} \label{eq:ecua}
h_\sigma(\nu) + \int \Delta_g ~{\rm d} \nu -P_{\Phi}(g) \int \tau ~{\rm d} \nu =0.
\end{equation}
 A direct application of  \cite[Theorem 2]{Sar01} gives
\begin{equation*}
h_\sigma(\nu) \leq \int \Big( P_\sigma(\Delta_g  -P_{\Phi}(g)\tau) - \Delta_g   + P_{\Phi}(g)\tau   \Big)~{\rm d} \nu
\end{equation*}
However, since by  \eqref{eq:ecua},
\begin{equation*}
h_\sigma(\nu) = -\int \left( \Delta_g \ + P_{\Phi}(g)  \tau \right) ~{\rm d} \nu,
\end{equation*}
we obtain
\begin{equation*}
P_\sigma(\Delta_g  -P_{\Phi}(g)\tau) =0.
\end{equation*}
This contradiction proves the statement.
\end{proof}

\begin{lema}
If $P_\sigma(\Delta_g -P_{\Phi}(g) \tau)=0$ and the potential $\Delta_g -P_{\Phi}(g) \tau$ is null-recurrent with corresponding infinite measure $\nu$ satisfying $\tau \in L^1(\nu)$ then
there exists an equilibrium measure for $g$.
\label{lem:null eq}
\end{lema}

\begin{proof}
In the proof of Lemma \ref{lem:neg base} we showed that if $\nu$ is the infinite RPF measure  associated to  $\Delta_g -P_{\Phi}(g) \tau$ satisfying $\tau \in L^1(\nu)$, and the flow is of finite entropy, then $\Delta_g \in  L^1(\nu)$. It is a consequence of
\cite[Theorem 2]{Sar01} that
\begin{equation*}
h_\sigma(\nu)= \int \Big( P_\sigma(\Delta_g  -P_{\Phi}(g)\tau) - \Delta_g   + P_{\Phi}(g)\tau   \Big) ~{\rm d} \nu.
\end{equation*}
Since $P(\Delta_g  -P_{\Phi}(g)\tau)=0$ we obtain
\begin{equation*}
h_\sigma(\nu) =- \int \Delta_g ~{\rm d} \nu  +P_{\Phi}(g) \int \tau~{\rm d}\nu.
\end{equation*}
That is
\begin{equation*}
P_{\Phi}(g) = \frac{h_\sigma(\nu)}{\int \tau ~{\rm d} \nu} + \frac{\int \Delta_g ~{\rm d} \nu}{\int \tau ~{\rm d} \nu}.
\end{equation*}
Therefore, for $\mu= R(\nu) \in \M_{\Phi}$,
\begin{equation*}
P_{\Phi}(g) = h_\Phi(\mu)+ \int g ~{\rm d} \mu,
\end{equation*}
so $\mu$ is an equilibrium measure for $g$, as required.
\end{proof}

\begin{proof}[Proof of  Theorem \ref{thm:flow eqstate}]
Case (a) of the theorem follows from \cite[Theorem 4]{bi1}.  Case (b) of the theorem follows by Lemma~\ref{lem:null eq}.

To complete the proof, we will show that cases (a) and (b) are the only cases in which there is an equilibrium measure for $g$.  Suppose that $g$ has a finite equilibrium measure $\mu$.   Then $\mu$ is of the form $R(\nu)$ where $\nu$ is a $\sigma$-invariant measure, which can be either finite or infinite and $\int\tau~{\rm d}\nu<\infty$. If $\nu$ is finite then it is an equilibrium measure for $\Delta_g-P_{\Phi}(g)\tau$ and $P(\Delta_g-P_{\Phi}(g)\tau)=0$. Thus $\Delta_g-P_{\Phi}(g)\tau$ is either recurrent or positive recurrent and so we are in case (a) or (b). If $\nu$ is infinite then it satisfies the variational principle for invariant measures (see Theorem \ref{thm:infiniteRPF}) and is a fixed point for $L_{\Delta_g-P_{\Phi}(g)\tau}^*$ and thus $\Delta_g-P_{\Phi}(g)\tau$ is null recurrent, as in case (b).
\end{proof}

\begin{rem}
Let $\mu$ be an equilibrium measure for the potential $g$ which can be written as
$\mu= R(\nu)$, where $\nu$ is the equilibrium measure for $\Delta_g -P_{\Phi}(g) \tau$. Techniques developed by Melbourne and T\"or\"ok \cite{mt} to obtain statistical limit theorems can be applied in this setting.  Indeed, let $\psi:Y \to \R$ be a zero mean potential, that is $\int \psi ~{\rm d}\mu =0$. Assume that $\tau \in L^a(\nu)$ and that $\psi \in L^b(\mu)$ with
\[\left( 1- \frac{1}{a} \right)\left( 1- \frac{1}{b} \right)\geq \frac{1}{2}.\]
 If $\Delta_{\psi}$ and $\tau$ satisfy the Central Limit Theorem then $\psi$ satisfies the Central Limit Theorem.
\end{rem}

\section{Inducing schemes} \label{sec:ind}
In this section we study the inducing procedure in the context of flows (compare to the discrete case in Section~\ref{ssec:induce discr}). While the inducing scheme theory is well developed in the context of maps, in our context has not been thoroughly studied. However, inducing was used in \cite[Section 6.2]{ij} to establish the existence of phase transitions for topologically mixing suspension flows and potentials satisfying certain growth conditions.   We will use the technique later to prove results on the recurrence properties of semi-flows as well as to help us understand specific examples.

Note that not only we obtain a  better combinatorial structure  when inducing, but also  if $\inf\{\tau(x):x\in C_a\}>0$ then the induced roof function $\overline{\tau}$ will be bounded away from zero.

We have already seen (see Remark \ref{rmk:Kac and Ab} ) that the integral of a potential with respect to a measure and the integral of the induced potential with respect to the lifted measure are related by the Kac formula. In the context of flows we also remark the following.

\begin{rem}[Kac's formula on flows] \label{rem:kac-flow}
Let $\nu \in \M_{\sigma}$ be ergodic and satisfy that $\nu(C_a)>0$. Denote by $\overline{\nu} \in \M_{\overline{\sigma}}$ its lift to the induced system as in \eqref{eq:proj meas}. Let $\phi:\Sigma \to \R$ be a potential of summable variations and let $\overline{\phi}$ be the corresponding induced potential as in \eqref{eq:overline phi}. The  case $\int r_a ~{\rm d} \overline{\nu} < \infty$  was considered in Remark \ref{rmk:Kac and Ab}. Note that if $\overline{\nu}\in \M_{\overline{\sigma}}$ with $\int\overline{\tau}~{\rm d}\overline\nu<\infty$ but $\int r_a~{\rm d}\overline{\nu}=\infty$ then the projection of $\nu$ is an infinite $\sigma$-invariant measure but $\mu=R(\nu)$ will be a finite $\Phi$-invariant measure. For a potential $g:Y\to\R$ we have
$$\int g~{\rm d}\mu=\frac{\int\overline{\Delta_g}~{\rm d}\overline{\nu}}{\int\overline{\tau}~{\rm d}\overline{\nu}}.$$
\end{rem}

Note that inducing corresponds to choosing a different base map to suspend the flow over. This can be thought of as considering a different cross-section to the flow.

The first important remark that we will make in this context is that the pressure for the flow can be computed using the induced system. This fact was implicit in \cite[Lemma 6.1]{ij}.
 Moreover,  if $\inf\{\tau(x):x\in C_a\}>0$ then the existence of equilibrium measures can be determined in the induced system.

 \begin{lema} \label{lem:indu-pre flow}
Let  $(Y, \Phi)$ be a suspension semi-flow defined over  a topologically mixing countable Markov shift $(\Sigma , \sigma)$ with a locally H\"older roof function $\tau$. Let $g:Y \to \R$ be such that $\Delta_g$ is locally H\"older.  Then
\begin{equation*}
P_{\Phi}(g) = \inf \left\{ s \in \R :P_{\overline{\sigma}}(\overline{\Delta_g- s\tau})\leq 0  \right\}.
\end{equation*}
Moreover if $\inf\{\tau(x):x\in C_a\}>0$ then $g$ has an equilibrium measure if and only if $P_{\overline{\sigma}}(\overline{\Delta_g-P_{\Phi}(g)\tau})=0$ and $\overline{\Delta_g-P_{\Phi}(g)\tau}$ has an equilibrium measure with respect to which $\overline{\tau}$ is integrable.
\end{lema}

\begin{proof}
By the definition of $P_{\Phi}(g)$ it suffices to show that
$$\sup\{s \in \R :P_{\sigma}(\Delta_g-s\tau )> 0\}=\sup\{s \in \R :P_{\overline{\sigma}}(\overline{\Delta_g- s\tau})>0\}.$$
To start let $t \in \R$ satisfy that $P_{\overline{\sigma}}(\overline{\Delta_g- t\tau})> 0$ and note that there must exist a compactly supported $\overline{\sigma}$-ergodic probability measure $\overline{\mu}$ such that
$$h_{\overline{\sigma}}(\overline{\mu})+\int \overline{\Delta_g}~{\rm d}\overline{\mu} -t \int \overline{\tau} ~{\rm d}\overline{\mu}>0.$$
By considering the projection of $\overline{\mu}$ and applying the variational principle we can deduce that $P_{\sigma}(\Delta_g -t\tau)>0$.

On the other hand if we fix  $t\in\R$ such that $P_{\sigma}(\Delta_g- t\tau)>0$ then we can find a compactly supported ergodic measure, $\mu$ such that $h_\sigma(\mu)+\int\Delta_g~{\rm d}\mu-t\int\tau~{\rm d}\mu>0$ and induce to observe that $P_{\overline{\sigma}}(\overline{\Delta_g-t\tau})>0$. The  first part of the result now follows.

To prove the second part  suppose that $P_{\overline{\sigma}}(\overline{\Delta_g-P_{\Phi}(g)\tau})=0$ and $\overline{\Delta_g-P_{\Phi}(g)\tau}$ has an ergodic equilibrium measure $\overline{\mu}$ such that $\int\overline\tau~{\rm d}\overline\mu<\infty$.
 Thus $\overline{\mu}$ can be pushed down to a (possibly infinite) $\sigma$-invariant measure $\mu$ where $\int\tau~{\rm d}\mu<\infty$ and
$$0=h_{\overline\sigma}(\overline{\mu})+\int\overline{\Delta_g-P_{\Phi}(g)\tau}~{\rm d}\overline{\mu}$$
which means
$$P_{\Phi}(g)=\frac{h_\sigma(\mu)}{\int\tau~{\rm d}\mu}+\frac{\int\Delta_g~{\rm d}\mu}{\int\tau~{\rm d}\mu}.$$
Therefore, $R(\mu)$ is a finite equilibrium measure for $g$.

On the other hand if $g$ has an equilibrium measure, then by Theorem \ref{thm:flow eqstate}  and Theorem \ref{thm:RPF} it has an equilibrium measure of the  form $R(\nu)$ where $\nu$  must be a RPF measure for $\Delta_g-P_{\Phi}(g)\tau$ (note that it can be infinite). Thus we can induce to yield $\overline{\nu}$ which will be an equilibrium measure for $\overline{\Delta_g-P_{\Phi}(g)\tau}$.
\end{proof}

\begin{rem}
In the proof of the second part of the lemma, the assumption  $\int\overline\tau~{\rm d}\overline\mu<\infty$ follows immediately in some cases.  For example, if $P_\Phi(g)\neq 0$ and $\Delta_g$ and $\tau$ are not asymptotically comparable (i.e., to rule out $|\int\overline\Delta_g ~{\rm d}\overline\mu|$ and $|\int\overline\tau ~{\rm d}\overline\mu|$ being simultaneously infinite, but $|\int\overline{\Delta_g-P_\Phi(g)\tau} ~{\rm d}\overline\mu|<\infty$), then the fact that $\overline\mu$ is an equilibrium measure for $\overline{\Delta_g-P_\Phi(g)\tau}$ implies that $|\int\overline{\Delta_g-P_\Phi(g)\tau} ~{\rm d}\overline\mu|<\infty$ and thus $\int\overline{\tau}~{\rm d}\overline{\mu}<\infty$.
\end{rem}

For use later, specifically in examples in Section~\ref{sec:rec flow}, we define
$$s_{\infty}:=\inf\{s:P(-s\overline{\tau})<\infty\}.$$
This number plays an important role in the thermodynamic formalism of the associated suspension flow.
 \begin{rem}\label{rmk:s infty prelim}
 We collect some basic facts about $s_\infty$ and $s\mapsto P_{\overline \sigma}(-s\overline\tau)$.
 \begin{enumerate}
 \item The constant $s_{\infty}=\infty$ if and only if $h(\Phi)=\infty$.
 \item Since $\tau\ge 0$, this is also true for the induced version.  Hence $s\mapsto P_{\overline \sigma}(-s\overline\tau)$ is a non-increasing function.  In particular, since $P_{\overline \sigma}(0)=\infty$, this means that $s_\infty\ge 0$.
\end{enumerate}
\end{rem}

\section{Recurrence and transience for suspension flows} \label{sec:rec flow}

In this section we extend the notions of recurrence and transience to potentials defined on suspension flows. This notions were given in the context of countable Markov shifts by Sarig in \cite{Sar99} and allow for the classification of a potential according to its recurrence properties. They have also been extended beyond the realm of Markov systems in \cite{it}.

We begin by defining the relevant partition function.
\begin{defi}
Let $g:Y \to \R$ be a potential such that $\Delta_g:\Sigma \to \R$ is of summable variations and $P_{\Phi}(g) < \infty$.
Given $C_{i_0}\in \Sigma$, let
$$Z_{n,\tau}(g, C_{i_0}):=\sum_{x\in C_{i_0},\ \phi_s(x, 0)=(x,0),\text{ for } n-1<s\le n}e^{\int_0^sg(\phi_t(0, x))~{\rm d}t}.$$
We say that $g$ is \emph{recurrent} if
$$
\sum_{n=1}^\infty Z_{n,\tau}(g- P_\Phi(g), C_{i_0})=\infty.$$
Otherwise, $g$ is  \emph{transient}.
\end{defi}
The following result establishes the relationship between recurrence and transience with equilibrium measures.
\begin{teo}
Let $g:Y \to \R$ be a potential such that $\Delta_g:\Sigma \to \R$ is of summable variations and $P_{\Phi}(g) < \infty$ and suppose $\tau$ satisfies  \eqref{eq:Hopf cond}.
\begin{enumerate}[label=({\alph*}),  itemsep=0.0mm, topsep=0.0mm, leftmargin=7mm]
\item The definition of recurrence is independent of the cylinder $C_{i_0}$.
\item $P_\Phi(g)=\lim_{n\to \infty}\frac1n\log Z_{n,\tau}(g, C_{i_0})$.
\item The potential $g$ is recurrent if and only if $\Delta_g-\tau P_\Phi(g)$ is recurrent and $P(\Delta_g  -P_{\Phi}(g)\tau)=0$.
\item If $g$ is recurrent then there exists a conservative measure $\nu_g$ which can be obtained as $R(\mu)$ where $\mu$ is the RPF measure for $\Delta_g-\tau P_\Phi(g)$.
\end{enumerate}
\label{thm:flow rec}
\end{teo}

\begin{rem}
The condition $P_\sigma(\Delta_g  -P_{\Phi}(g)\tau)=0$ in (c) is crucial.  In Example~\ref{eg:Hof nlogn}  below we construct a case where  $\Delta_g-\tau P_\Phi(g)$ is recurrent, but $P_\sigma(\Delta_g  -P_{\Phi}(g)\tau)<0$ and thus $g$ is transient.
\end{rem}

\begin{proof}
The proof of (a) follows from the proof of (c).
The proof of (b) follows from  \cite{jkl}.

For part (c), we will shortly prove the following:

\begin{claim}
$$\sum_{n=1}^\infty Z_{n,\tau}(g- P_\Phi(g), C_{i_0})=\sum_{n=1}^\infty Z_n(\Delta_g-\tau P_\Phi(g), C_{i_0}).$$
\label{cl:Zns same}
\end{claim}

This can be added to the following claim, which follows immediately from the definition of pressure being the radius of convergence of the power series defined by partition functions.

\begin{claim}
$P_\sigma(\Delta_g  -P_{\Phi}(g)\tau)<0$ implies that  $\sum_{n=1}^\infty Z_n(\Delta_g-\tau P_\Phi(g), C_{i_0})<\infty$.
\end{claim}

Together the claims show that $g$ being recurrent implies that $\Delta_g-\tau P_\Phi(g)$ is recurrent and $P(\Delta_g  -P_{\Phi}(g)\tau)=0$.  Moreover, Claim~\ref{cl:Zns same} alone shows that if  $\Delta_g-\tau P_\Phi(g)$ is recurrent and $P(\Delta_g  -P_{\Phi}(g)\tau)=0$ then $g$ is recurrent.  So (c) holds,

\begin{proof}[Proof of Claim~\ref{cl:Zns same}]
First for $x\in C_{i_0}$ such that $\phi_s(x, 0)=(x, 0)$ where $n-1<s\le n$, set $k_{x,n}$ to be the number of times that $(x, 0)$ has returned to the base at time $s$, i.e., let  $k \in \N$ such that $\tau^k(x)=s$.
Then
\begin{align*}
Z_{n,\tau}(g- P_\Phi(g), C_{i_0})&=\sum_{x\in C_{i_0},\ \phi_s(x, 0)=(x,0),\text{ for } n-1<s\le n}e^{\int_0^s(g- P_\Phi(g))(\phi_t(0, x)) ~{\rm d}t}\\
&=\sum_{x\in C_{i_0},\ \phi_s(x, 0)=(x,0),\text{ for } n-1<s\le n}e^{\sum_{i=0}^{k_{x, n}-1}\int_{\tau^i(x)}^{\tau^{i+1}(x)}(g- P_\Phi(g))(\phi_t(0, x)) ~{\rm d}t}\\
&=\sum_{x\in C_{i_0},\ \phi_s(x, 0)=(x,0),\text{ for } n-1<s\le n}e^{\left(\sum_{i=0}^{k_{x, n}-1}\Delta_g(\sigma^i(x))\right)-sP_\Phi(g)}.
\end{align*}

Since
$$Z_n(\Delta_g-\tau P_\Phi(g), C_{i_0})= \sum_{\{x\in C_{i_0} :\sigma^nx=x\}}e^{\left(\sum_{i=0}^{n-1}\Delta_g(\sigma^i(x))\right)-n P_\Phi(g)},$$ each term $e^{\left(\sum_{i=0}^{k_{x, n}-1}\Delta_g(\sigma^i(x))\right)-sP_\Phi(g)}$ in the sum for $Z_{n,\tau}(g- P_\Phi(g), C_{i_0})$ is counted as part of the sum  for $Z_{k_{x, n}}(\Delta_g-\tau P_\Phi(g), C_{i_0})$.  Thus the sum $\sum_{n=1}^\infty Z_{n,\tau}(g- P_\Phi(g), C_{i_0})$ is simply a reordering of the series of positive terms  $\sum_{n=1}^\infty Z_n(\Delta_g-\tau P_\Phi(g), C_{i_0})$.
\end{proof}

For (d), the existence of the measure described follows from (c) plus the RPF Theorem (see Theorem \ref{thm:RPF})  applied to $\Delta_g-\tau P_\Phi(g)$.  Conservativity follows from the fact that since the RPF measure is conservative then so is $R(\mu)$.
\end{proof}

In order to give a criterion for the existence of equilibrium measures for the flow, we consider the following,
\begin{defi}
Let $g:Y \to \R$ be a potential such that $\Delta_g:\Sigma \to \R$ is of summable variations and $P_{\Phi}(g) < \infty$. Let
\begin{align*}
P_n:= \{x\in C_{i_0}:\exists s\in (n-1, n] & \text{ s.t. } \phi_s(x, 0)=(x, 0),\\
&\text{ but } \phi_t(x, 0)\notin Y\cap(C_{i_0}\times\R) \text{ for any } t\in (0,s)\}.
\end{align*}
For $x\in P_n$, let $s_{x}>0$ be minimal such that $\phi_{s_{x}}(x, 0)=(x, 0)$.  Now let
$$Z_{n, \tau}^*(g, C_{i_0}):=\sum_{x\in P_n} e^{\int_0^{s_{x}}g(\phi_t(x)) ~{\rm d}t}.$$
\end{defi}

\begin{rem}
For each $x\in P_n$ there exists $n_x\in \N$ such that $s_x=\tau^{n_x}(x)$, so for a potential $\tilde g$ such that  $\Delta_{\tilde g}:\Sigma \to \R$ is of summable variations, $P(\Delta_{\tilde g}) < \infty$ and $\overline{\Delta_{\tilde g}}$ is locally H\"older we have
 $$\int_0^{s_x} \tilde g(\phi_t)~{\rm d}t=S_{n_x}\Delta_{\tilde g}=\overline{\Delta_{\tilde g}}(x).$$  Let $Y_x$ be the $(n_x-1)$-cylinder $Y$ around $x$ w.r.t. the dynamics $\sigma$, so $\sigma^{n_x}(Y)=C_{i_0}$.  This set is a 1-cylinder for the induced map $\overline\sigma$.  We will use the fact below that if $\hat\mu$ is a Gibbs measure for $\overline{\Delta_{\tilde g}}$ then $\overline\mu(Y)$ is uniformly comparable to $e^{ \overline{\Delta_{\tilde g}}(x)}$.  Note that this term is a summand in the sum for $Z_{n,\tau}^*$.
\label{rmk:Gibbs flow}
\end{rem}

\begin{defi}
Let $g:Y \to \R$ be a potential such that $\Delta_g:\Sigma \to \R$ is of summable variations and $P_{\Phi}(g) < \infty$. Suppose that $g$ is recurrent.  If
$$\sum_nnZ_{n, \tau}^*(g)e^{-nP_\Phi(g)}<\infty$$
we say that $g$ is \emph{positive recurrent}. If
$$\sum_nnZ_{n, \tau}^*(g)e^{-nP_\Phi(g)}=\infty$$
we say that $g$ is \emph{null recurrent}.
\end{defi}

Note that due to the topological mixing of $\sigma$, the definition above is independent of the cylinder $C_{i_0}$.

\begin{teo}
Suppose that $\tau$ satisfies  \eqref{eq:Hopf cond} and $h(\Phi)<\infty$.  If $g:Y \to \R$ is a positive recurrent potential, then there exists an equilibrium measure $\nu_g$.
\end{teo}
\begin{proof}
Define $\overline{\Delta_g-\tau P_\Phi(g)}$ to be the induced version of $\Delta_g-\tau P_\Phi(g)$ on $C_{i_0}$.
 Also define $\overline\tau$ to be the induced roof function.  Since $g$ is recurrent and, by Theorem~\ref{thm:flow rec}, $P_\sigma(\Delta_g-\tau P_\Phi(g))=0$, the proof of \cite[Lemma 3]{Sar01a} implies that there is a Gibbs measure $\overline\mu$ for $\overline{\Delta_g-\tau P_\Phi(g)}$, which projects to the flow if $\int\overline\tau~{\rm d}\overline\mu<\infty$ (see Remark~\ref{rmk:Kac and Ab}).  Since by the same result (see also \cite{Sar03}), $\overline\mu$ is a Gibbs measure, as in Remark~\ref{rmk:Gibbs flow} the value of $\int\overline\tau~{\rm d}\overline\mu$  can be bounded by a constant times $\sum_nnZ_{n, \tau}^*(g)e^{-nP_\Phi(g)}$.  Since $g$ is positive recurrent, these values are bounded and so $\int\overline\tau~{\rm d}\overline{\mu}<\infty$. This means $\overline\mu$ does indeed project to a $\Phi$-invariant probability measure $\nu_g$.  Since we're assuming that  $h(\Phi)<\infty$, this implies that $h_\Phi(\nu_g)<\infty$, and moreover, using finiteness of the integral of $\overline\tau$ again, that $h_{\overline\sigma}(\overline\mu)<\infty$. Therefore $\overline{\mu}$ is an equilibrium state for $\overline{\Delta_g-\tau P_\Phi(g)}$ which projects to the measure $\nu_g$ which is, in turn, an equilibrium state for $g$.
\end{proof}

\begin{rem}
In the spirit of \cite{it}, we now characterise recurrence in a way which extends beyond semi-flows over
  finite shifts.  Suppose that $\tau$ satisfies  \eqref{eq:Hopf cond} and $h(\Phi)<\infty$.  Let $g:Y \to \R$ be a potential such that $\Delta_g:\Sigma \to \R$ is of summable variations and $P_{\Phi}(g) < \infty$. Sumarising,  we say that $g$ is
\begin{enumerate}
\item \emph{Positive recurrent} if it has an equilibrium measure;
\item \emph{Null-recurrent} if $P_\sigma(\Delta_g -P_{\Phi}(g) \tau ) = 0$, the potential $ \Delta_g -P_{\Phi}(g) \tau$ is recurrent with corresponding measure $\nu$ and $\int \tau ~{\rm d} \nu
= \infty$;
\item \emph{Transient} in any other case.
\end{enumerate}
In particular, our definitions extend the corresponding ones given by Sarig \cite{Sar99} for countable Markov shifts. Note that these definitions are independent of the particular choice of base dynamics, so for example still hold if we take a first return map to part of $\Sigma$.
\end{rem}

\section{The renewal flow} \label{sec:reflow}

In this section we define and study a class of suspension flows of particular interest, since they serve as symbolic models for a class of flows belonging to the boundary of hyperbolic flows.
We study the corresponding thermodynamic formalism establishing conditions for the existence of equilibrium measures and phase transitions.

\subsection{The renewal shift and the renewal flow}
For the alphabet $\N_0$, consider the
transition matrix $A=(a_{ij})_{i,j \in \N_0}$ with $a_{0,0}= a_{0,n}=
a_{n,n-1}=1$ for each $n \ge 1$ and with all other entries equal to
zero. The \emph{renewal shift} is the Markov shift
$(\Sigma_R, \sigma)$ defined by the transition matrix $A$, that~is,
the shift map $\sigma$ on the space
\[
\Sigma_R = \left\{ (x_i)_{i \ge 0} : x_i \in \N_0 \text{ and } a_{x_i
x_{i+1}}=1 \text{ for each } i \geq 0\right\}.
\]

\begin{rem} \label{rmk:full}
Let $(\Sigma_2 ,\sigma)$ be the  full-shift on the
alphabet $\{0,1\}$. There exists a topological conjugacy between the
renewal shift $(\Sigma_R,\sigma)$ and $(\Sigma_2 \setminus
\bigcup_{i=0}^{\infty} \sigma^{-i}(\overline{0}) ,\sigma)$, where
$\overline{0}=(000 \cdots)$. Indeed, denote by $(0 \cdots 01)_n$ the
cylinder $C_{0 \cdots 01}$ with $n$ zeros, and consider the alphabet
$\{ (0 \cdots 01)_n : n \ge 1 \} \cup \{C_0\}$. The possible
transitions on this alphabet are
\[
(0 \cdots 01)_n \to (0 \cdots 01)_{n-1},\ C_0 \to C_0, \text{ and }
C_0 \to (0 \cdots 01)_n \text{ for } n \ge 1.
\]
Note that this is simply a recoding of $(\Sigma_2 \setminus
\bigcup_{i=0}^{\infty} \sigma^{-i}(\overline{0}) ,\sigma)$.
\end{rem}

 Let $\mathcal{R}$ be the class of functions $\phi\colon
\Sigma_{R} \to\R$ such that:
\begin{enumerate}
\item the function $\phi$ has summable variation and is bounded from above;
\item the function $\phi$ has finite Gurevich pressure;
\item the induced function $\overline{\phi}$ is  locally H\"older continuous.
\end{enumerate}
We observe that $\mathcal{R}$ includes the class of H\"older
continuous functions. Nevertheless, there are non-H\"older continuous functions that belong
to~$\mathcal{R}$.

Thermodynamic formalism is well understood in this setting for potentials of summable variations. Indeed, Sarig \cite{Sar01a} proved the following (the version of this result for every $q \in \R$, and not just for positive values, appears in \cite[Proposition 3]{bi2}).

\begin{prop}\label{prop:BI ana}
Let $(\Sigma_R, \sigma)$ be the renewal shift. For each bounded
$\phi \in \mathcal{R}$ there exist $q_{c}^+ \in (0, +\infty]$ and
$q_{c}^- \in [-\infty, 0)$ such that:
\begin{enumerate}[label=({\alph*}),  itemsep=0.0mm, topsep=0.0mm, leftmargin=7mm]
\item $q \mapsto P_\sigma(q\phi)$ is strictly convex and real analytic in~$(q_{c}^{-},
q_{c}^{+})$.
\item
$P_\sigma(q\phi)=mq$ for $q<q_{c}^{-}$, and $P_G(q\phi)=Mq$ for
$q>q_{c}^{+}$. Here $m:= \inf \left\{ \int_{\Sigma_R}\phi \, d \mu : \mu \in
\mathcal{M}_R \right\}$ and $M:= \sup \left\{ \int_{\Sigma_R}\phi \, d \mu : \mu \in
\mathcal{M}_R \right\}.$
\item At $q_{c}^{-}$ and $q_{c}^{+}$ the function $q \mapsto P_\sigma(q\phi)$
is continuous but not analytic.
\item For each $q \in (q_c^-,q_c^+)$ there is a unique equilibrium measure
$\mu_{q}$ for~$q\phi$.
\item For each $q\not\in[q_c^-,q_c^+]$ there
is no equilibrium measure for $q\phi$ and this potential is transient.
\item The critical values $q_{c}^+$ and $q_{c}^-$ are never
simultaneously finite.
\end{enumerate}
\end{prop}

To reiterate, if $ q \in (q_{c}^{-},q_{c}^{+})$ the potential $q \phi$ is positive recurrent and for $q<q_c^{-}$ or $q>q_c^{+}$ the potential $q \phi$ is transient.
At the critical values the potential can have any  recurrence mode: $q_c \phi$ can be positive recurrent, null-recurrent or transient (see \cite[Example 2]{Sar01a}).

\begin{defi}
Let $\tau:\Sigma_R \to \mathbb{R}$ be a roof function with $\tau \in \mathcal{R}$ and
 satisfying \eqref{eq:Hopf cond}.
The suspension semi-flow $\Phi_R=\varphi_t(x,s)$ defined in the canonical way on the (non-compact) space
\begin{equation*}
Y_R= \{ (x,t)\in \Sigma_R \times \R \colon 0 \le t \le\tau(x)\}.
\end{equation*}
is called a \emph{renewal semi-flow}.
\end{defi}

If $\lim_{x \to \overline0}\tau(x)=0$ then we can think of this flow as one having a \emph{cusp} at $(\overline{0},0)$. This, of course, has several dynamical consequences. For instance:

\begin{eje}[An infinite entropy renewal flow]  Here we present an example of a semi-flow where the presence of a cusp causes the flow to have infinite topological entropy.   Clearly for this to be interesting, the base dynamics should have finite entropy: we consider the renewal shift, which has topological entropy $\log 2$.

Consider the renewal semi-flow, with roof function $\tau$ to be defined later.
By  Propositions ~\ref{prop:Abr} and \ref{prop:Sav}, $$h(\Phi)=\sup_{\nu\in \mathcal{E}_\sigma}\left\{\frac{h_\sigma(\nu)}{\int\tau ~{\rm d}\nu}\right\}.$$
Consider the induced system $(\overline\Sigma, \overline\sigma, \overline \tau)$ given by the first return map to $C_0$.  As usual, for each $n\in \N$, denote the domain with first return time $n$ by $X_n$.  Then by the Abramov formula and by approximating measures by compactly supported ones,
 $$h(\Phi)=\sup_{\overline\nu\in \mathcal{E}_{\overline\sigma}(\overline\tau)}\left\{\frac{h_{\overline\sigma}(\overline\nu)}{\int\overline\tau ~{\rm d}\overline\nu}\right\}=\sup_{\overline\nu\in \mathcal{E}_{\overline\sigma}}\left\{\frac{h_{\overline\sigma}(\overline\nu)}{\int\overline\tau~{\rm d}\overline\nu}\right\}$$
where $\mathcal{E}_{\overline\sigma}(\overline\tau)$ is the space of ergodic $\overline{\sigma}$-invariant  measures for which $\tau$ is integrable.
 (We don't actually use the second equality here.)
If $\overline\nu$ is a Markov measure (see \cite[p.22]{wa} for the definition) for $(\overline\Sigma, \overline\sigma)$, then
$$\frac{h_{\overline\sigma}(\overline\nu)}{\int\overline\tau~{\rm d}\overline\nu}=-\frac{\sum_n\nu(X_n)\log\nu(X_n)}{\sum_n\nu(X_n)s_n},$$
where $s_n=\overline\tau|_{X_n}$.

Setting
\begin{equation*}
\tau(x)=\begin{cases}
\log \log (1+e) & \text{ if } x\in C_0\\
\log \log (1+e+n)-\log \log (1+e+n-1) & \text{ if } x\in  C_n\text{ for }n\geq 1,
\end{cases}
\end{equation*}
we obtain $s_n=\log \log (e+n)$.  So in this case,
$$\frac{h_{\overline\sigma}(\overline\nu)}{\int\overline\tau~{\rm d}\overline\nu}=-\frac{\sum_n\nu(X_n)\log\nu(X_n)}{\sum_n\nu(X_n)\log\log(e+n)}.$$
So for example, for $N\in \N$, the measure $\overline\nu_N$ giving mass $1/N$ to $X_n$ if $1\le n\le N$ and zero mass otherwise has
$$\frac{h_{\overline\sigma}(\overline\nu_N)}{\int\overline\tau~{\rm d}\overline\nu_N}=\frac{N\log N}{\sum_{n=1}^N\log\log(e+n)}\ge \frac{\log N}{\log\log(e+N)}\to \infty \text{ as } N\to \infty.$$
Thus $h(\Phi)=\infty$.  This argument implies that we have some freedom to alter $\tau$, but so long as it is chosen so that $\log N/s_N\to \infty$ as $N\to \infty$, the entropy of the flow will still be infinite.
\label{eg:inf ent}
\end{eje}

\subsection{Equilibrium measures for the renewal flow}

In the next proposition we characterise bounded potentials having equilibrium measures.

\begin{prop} \label{prop:ren-eq}
Let   $\Phi_R$ be a renewal semi-flow of finite entropy and $g: Y_R \to \R$ a bounded potential such that $\Delta_g : \Sigma_R \to \R$ is locally H\"older. There exists an equilibrium measure for $g$  if and only if  the potential $\Delta_g -P_{\Phi}(g) \tau$ is recurrent, and the  RPF  measure $\nu_g$  has $\int \tau ~{\rm d} \nu_g < \infty$.  Here  $\mu=\nu_g\times m)|_Y/((\nu_g\times m)(Y))$. Moreover, if the potential $g$ has an equilibrium measure, then it is unique.
\end{prop}

\begin{proof}
We would like to apply Theorems~\ref{thm:flow eqstate} and then Theorem~\ref{thm:uni} here.
First we will show that the equation in the variable $ q \in \R$ given by
\[ P_\sigma(\Delta_g -q(g) \tau) =0,\]
always has a root. Indeed, since the potential $g$ is bounded there exist constants  $K_1, K_2 \in \R$ such that
$K_2 \tau \leq	\Delta_g \leq K_1 \tau.$ Therefore,
\begin{equation*}
 P_\sigma((K_2- q) \tau) \leq  P_\sigma(\Delta_g -q \tau) \leq  P_\sigma((K_1- q) \tau).
  \end{equation*}
It is a direct consequence of Proposition \ref{prop:BI ana} and of the fact that $h(\Phi)< \infty$ that $P_\sigma(-h(\Phi) \tau) =0$. Therefore there exist $q_1, q_2 \in \R$ such that
\[ 0 \leq P_\sigma((K_2- q_1) \tau) \leq  P_\sigma(\Delta_g -q_1 \tau) <\infty, \]
and
\[ P_\sigma(\Delta_g -q_2 \tau) \leq  P_\sigma((K_1- q_2) \tau) \leq 0.\]
Since the pressure is a continuous function of the variable $q$ we obtain the desired result.

By virtue of Theorem~\ref{thm:flow eqstate}, if the potential
$\Delta_g -P_{\Phi}(g) \tau$ is transient then there are no equilibrium measures for $g$.

If $\Delta_g -P_{\Phi}(g) \tau$ is positive recurrent then there exist a measure $\nu_g \in \Sigma_R$ which is an equilibrium measure for $\Delta_g -P_{\Phi}(g) \tau$ . Since the roof function $\tau\in \mathcal{R}$, it is bounded, so $\tau \in L^1(\nu_g)$, therefore there exists an equilibrium measure for $g$.

 The remaining case is when the potential  $\Delta_g -P_{\Phi}(g) \tau$ is null recurrent. Denote by $\nu_g$ the  corresponding infinite RPF measure. Theorem \ref{thm:flow eqstate} together with the fact that  $\tau \in L^1(\nu_g)$ yield the desired result.

The uniqueness of the equilibrium measure is a direct consequence of Theorem  \ref{thm:uni}.
  \end{proof}

\begin{rem}
If the renewal flow $\Phi_R$ has infinite entropy, as in Example~\ref{eg:inf ent}, then it is a direct consequence of the variational principle that bounded potentials  on $Y_R$ do not have equilibrium measures.
 \end{rem}

\begin{rem}
In the proof of Proposition~\ref{prop:ren-eq} we obtained that  if $\Delta_g -P_{\Phi}(g) \tau$ is positive recurrent and $\nu_g$ is the corresponding measure then $\tau \in L^1(\nu_g)$. Therefore, in the positive recurrent case that assumption is not needed.
 \end{rem}

\begin{rem}
Let   $\Phi_R$ be a renewal flow of finite entropy and $g: Y_R \to \R$ a bounded potential such that $\Delta_g : \Sigma_R \to \R$ is locally H\"older . Recall that  locally H\"older potentials of finite pressure defined on the full-shift on a countable alphabet are positive recurrent (see \cite[Corollary 2]{Sar03}). By Kac's formula obtained in \cite[Lemma 3]{Sar01a} we have that if $P_{\overline\sigma}(\overline{\Delta_g -P_{\Phi}(g) \tau})=0$ then  $\Delta_g -P_{\Phi}(g) \tau$ is recurrent. Denote by $\nu_g$ the corresponding equilibrium measure. So as above, to find an equilibrium measure for $g$, it suffices to check that for the RPF measure $\nu_g$, we have $\tau \in L^1(\nu_g)$.
\end{rem}

\subsection{Measures of maximal entropy: Hofbauer type roof functions}
\label{sec:Hof roof}

In this subsection we look at when suspension flows over the renewal shift have measures of maximal entropy. We will assume that
$$\lim_{n\to\infty}\sup_{x\in C_n}\{\tau(x)\}=0.$$
Note that this is not in conflict with \eqref{eq:Hopf cond}, and indeed the renewal structure forces \eqref{eq:Hopf cond} to be true for roof functions where $\inf_{x\in C_0}\{\tau(x)\}>0$.
Recall that we induce on the cylinder $C_0$ to obtain $(\overline\Sigma, \overline\sigma)$ with corresponding roof function $\overline\tau$.
 By Lemma \ref{lem:indu-pre flow} it follows that there exists a measure of maximal entropy if and only if the roof function $\tau$ satisfies that $P_{\overline{\sigma}}(-h(\Phi)\overline{\tau})=0$ and $-h(\Phi)\overline{\tau}$ has an equilibrium measure. This means that in particular if $P_{\overline{\sigma}}(-s_{\infty}\overline{\tau})>0$ then there exists a measure of maximal entropy; however if $P_{\overline{\sigma}}(-s_{\infty}\overline{\tau})<0$ there exists no measure of maximal entropy.

We can also say that if $P_{\overline{\sigma}}(-s_{\infty}\overline{\tau})>0$ then $-h(\Phi)\tau$ will be positive recurrent and will have an equilibrium measure if and only if $\int r_0~{\rm d}\overline{\mu}<\infty$ where $\overline{\mu}$ is the equilibrium measure for $-h(\Phi)\overline{\tau}$. Otherwise $-h(\Phi)\tau$ will be null-recurrent. Specific cases of this where $-h(\Phi)\tau$ is null-recurrent are given below in Lemma \ref{lem:MP}.

In the case where $P_{\overline{\sigma}}(-s_{\infty}\overline{\tau})=0$ and thus $h(\Phi)=s_{\infty}$ then it is possible that the Gibbs measure $\overline{\mu}$ for $-s_{\infty}\overline{\tau}$ will satisfy that $\int\overline{\tau}~{\rm d}\overline{\mu}=\infty$. In this case $-h(\Phi)\tau$ is null recurrent and there is no measure of maximal entropy for $\Phi$. Finally if $P_{\overline{\sigma}}(-s_{\infty}\overline{\tau})<0$ and thus $h(\Phi)=s_{\infty}$ then $P(-s_{\infty}\tau)=0$ and $-s_\infty\tau$ is transient.  We can adapt the techniques used by Hofbauer in \cite{ho} to produce examples where there is no measure of maximal entropy and $-h(\Phi)\tau$ is either transient or null-recurrent (In this setting if $-h(\Phi)\tau$ is positive recurrent then there must be a measure of maximal entropy for $\Phi$.)

\begin{eje}\label{eg:Hof nlogn}
We fix $k>0$ such that $\sum_{n=0}^{\infty}\frac{1}{(n+k)(\log(n+k))^2}<1$. Then for $n\geq 0$ let $a_n:=\frac{1}{(n+k)(\log(n+k))^2}$
 and consider the locally constant potential $\tau:\Sigma_R \to \R$ defined on each cylinder $C_n$ with $n \geq 1$ by $\tau|_{C_n} =-\log\left(\frac{a_n}{a_{n-1}}\right)$ and with $\tau|_{C_0}=-\log a_0$. This gives that the induced roof function $\overline{\tau}$ will also be locally constant with $\overline{\tau}|_{\overline{C_n}}=-\log a_n$ for $n\geq 0$. Thus
$$P_{\overline{\sigma}}(-t\overline{\tau})=\log\left(\sum_{n=0}^{\infty}\left(\frac{1}{(n+k)(\log(n+k))^2}\right)^t\right)$$
when this is finite, and otherwise $P_{\overline{\sigma}}(-t\overline{\tau})=\infty$. So we have that $s_{\infty}=1$ and $P_{\overline{\sigma}}(-\overline{\tau})<0$. Thus $h(\Phi)=1$ and there is no measure of maximal entropy for $\Phi$ and $-\tau$ is transient. This example also
relates to the construction of infinite iterate function systems with no measure of maximal dimension considered by Mauldin and Urba\'{n}ski in
\cite{mu1}.

Now consider the alternative case when $a_n$ is the same as above for $n\ge 1$, but $a_0=1-\sum_{n=1}^{\infty}a_n$. Define $\tau$ as above and note that we now have that $P(-\overline{\tau})=0$. Moreover,
$$-\sum_{n=0}^{\infty}a_n\log a_n=\infty$$
and so the Gibbs measure $\overline{\mu}$ for $\overline{\tau}$ satisfies $\int\overline{\tau}~{\rm d}\overline{\mu}=\infty$ and we are in the case when there is no measure of maximal entropy for $\Phi$ and $-\tau$ is null-recurrent.
\end{eje}

\subsection{Phase transitions for the renewal flow}  Bowen and Ruelle \cite{br} showed that the pressure function, $t \mapsto P_{\Phi}(t g)$,
for suspension flows defined over (finite state) sub-shifts of finite type with H\"older roof function is real analytic when considering potentials $g$ such that $\Delta_g$ is  H\"older.  In particular, the pressure is real analytic for Axiom A flows. Note that since the pressure is convex it is differentiable at every point of the domain, except for at  most a countable set.  We say that the pressure exhibits a \emph{phase transition}  at the point $t_0 \in \R$ if the function $P_{\Phi}(t g)$ is \emph{not} real analytic at $t=t_0$. In \cite{ij} the regularity of the pressure was studied and conditions in order for the pressure to be real analytic or to exhibit phase transitions were found in the context of BIP shifts in the base (see \cite{ij} for precise definitions, but roughly speaking these are shifts that combinatorially are close to the full-shift).  In the discrete time setting, the renewal shift is fairly well understood (see
Proposition \ref{prop:BI ana} and \cite{Sar01a}). The pressure exhibits at most one phase transition after which the pressure takes the form $P(t \psi)= At$. Phase transitions of this type are called phase transition of \emph{zero entropy}, since the line $At$ passes through zero. Recently, in the context of maps phase transitions of \emph{positive entropy} have been constructed (in this cases the pressure function takes the form $P(t \psi)= At+B$, with $B \neq 0$). Indeed, examples have been obtained in \cite{dgr, it} and  perhaps most comprehensively in \cite{bt}, where $B= h_{top}(T)$.

For the renewal flow the situation is richer than in the renewal shift setting.
 In particular, the pressure function can exhibit two phase transitions (see Example \ref{eg:2phase}) as opposed to the discrete time case, where at  most there exists one phase transition.  Moreover it can exhibit phase transitions of zero and positive entropy depending on the value of $s_{\infty}$ (see Section~\ref{sec:ind} for a precise definition). Indeed, the case in which $s_{\infty}=0$ can be thought of as zero entropy phase transitions  and $s_{\infty}>0$ correspond to positive entropy phase transitions (see Example \ref{eg:2phase} for a positive entropy phase transition).  In this sub-section we establish conditions in order for the pressure function $t \mapsto P_{\Phi}(tg)$ to be real analytic or to exhibit phase transitions.

We fix a positive function $\tau\in-\mathcal{R}$
 as our roof function and let $g:Y_{R}\to\R$ be a function such that $\Delta_g\in\pm\mathcal{R}$.  In addition we will assume that  there exists $\alpha\in\R$ such that
\begin{equation}
\lim_{n\to\infty}\sup_{x\in C_n}\left\{\frac{\Delta_g(x)}{\tau(x)}\right\}=\lim_{n\to\infty}\inf_{x\in C_n}\left\{\frac{\Delta_g(x)}{\tau(x)}\right\}=\alpha.
\label{eq:alpha}
\end{equation}
Recall that we proved in Lemma \ref{lem:indu-pre flow} that we can calculate the pressure of $g$ using the induced map. We have that
$$P_{\Phi}(g)=\inf\{s:P_{\overline{\sigma}}(\overline{\Delta_g-s\tau})\leq 0\}.$$

\begin{rem}
Recall that $s_{\infty}=\inf\{s:P(-s\overline{\tau})<\infty\}$. Under condition \eqref{eq:alpha}, and assuming $h(\Phi)<\infty$, there must exist a sequence of measures $(\overline\nu_n)_n$ such that $\int\overline\tau~{\rm d}\overline\nu_n\to\infty$.  This follows from the variational principle and the two observations that $h(\overline\sigma)=\infty$; and for each $s>s_\infty\ge 0$ there exists $C>0$ such that $h_{\overline\sigma}(\overline\nu)-s\int \overline\tau~{\rm d}\overline\nu<C$ for any $\overline\sigma$-invariant measure $\overline\nu$.

\label{rmk:s infty}
\end{rem}

\begin{lema}\label{lem:critvalue}
For $t\in\R$ we have that
$$\inf\{s:P_{\overline{\sigma}}((\overline{t\Delta_g-s\tau}))<\infty\}=s_{\infty}+t\alpha.$$
\end{lema}

\begin{proof}
For $t=0$ the result is obvious so we will assume throughout the proof that $t\neq 0$.
If we let $s<s_{\infty}+t\alpha$ then since, as in Remark~\ref{rmk:s infty}, there exists a sequence of $\overline{\sigma}$-invariant probabily measures $\overline{\nu_n}$ such that $\lim_{n\to\infty}\int\overline{\tau}~{\rm d}\overline{\nu_n}=\infty$ and thus $\lim_{n\to\infty}\frac{\int\overline{\Delta_g}~{\rm d}\overline{\nu_n}}{\int\overline{\tau}~{\rm d}\overline{\nu_n}}=\alpha$ it follows that
$P_{\overline{\sigma}}(\overline{-s\tau+\Delta_g})=\infty$. Hence $\inf\{s:P_{\overline{\sigma}}(\overline{t\Delta_g-s\tau})<\infty\}\ge s_{\infty}+t\alpha$.

Let $\epsilon>0$ and let $\overline{\nu}$ be a $\overline\sigma$-invariant probability measure where $\int\overline{\tau}~{\rm d}\overline{\nu}<\infty$ and $\left|\frac{\int\overline{\Delta_g}~{\rm d}\overline{\nu}}{\int\overline{\tau}~{\rm d}\overline{\nu}}-\alpha\right|<\frac{\epsilon}{2}.$
We then have that
\begin{align*}
h_{\overline\sigma}(\overline{\nu})-(s_{\infty}+t\alpha+|t\epsilon|)\int\overline{\tau}~{\rm d}\overline{\nu}+t\int\overline{\Delta_g}~{\rm d}\overline{\nu}
& \leq h_{\overline\sigma}(\overline{\nu})-\left(s_{\infty}+\frac{|t\epsilon|}{2}\right)\int\overline{\tau}~{\rm d}\overline{\nu}\\
&\leq P_{\overline{\sigma}}(-(s_{\infty}+|t\epsilon|/2)\overline{\tau})<\infty.
\end{align*}
On the other hand, \eqref{eq:alpha} implies that given $\epsilon>0$, there exists $C>0$ such that \begin{equation}
\left|\frac{\int\overline{\Delta_g}~{\rm d}\overline{\nu}}{\int\overline{\tau}~{\rm d}\overline{\nu}}-\alpha\right|>\frac{\epsilon}{2}
\label{eq: s inf bit}
\end{equation}
implies that $\int\overline{\tau}~{\rm d}\overline{\nu}<C$.  Again using \eqref{eq:alpha}, this means that $h_{\overline\sigma}(\overline{\nu})-(s_{\infty}+t\alpha+|t\epsilon|)\int\overline{\tau}~{\rm d}\overline{\nu}+t\int\overline{\Delta_g}~{\rm d}\overline{\nu}$ is bounded above, with the bound depending only on $s_\infty, t, \alpha, \epsilon$  and condition \eqref{eq: s inf bit}.  Therefore the variational principle implies that $P_{\overline{\sigma}}(\overline{t\Delta_g-(s_\infty+t\alpha+\delta)\tau})<\infty$ for any $\delta>0$.  Hence $\inf\{s:P_{\overline{\sigma}}(\overline{t\Delta_g-s\tau})<\infty\}\le s_{\infty}+t\alpha$, as required.
\end{proof}

The lemma immediately implies that $P_{\Phi}(tg) \geq s_{\infty}+t\alpha$. We now let
$$I:=\{t:P_{\overline{\sigma}}(\overline{-(s_\infty+t\alpha)\tau+t\Delta_g})\leq 0\}$$
and note that by the convexity of pressure this is either the empty set or an interval.

\begin{prop}\label{PT}
The regularity of the pressure function is given by
\begin{enumerate}[label=({\alph*}),  itemsep=0.0mm, topsep=0.0mm, leftmargin=7mm]
\item
For $t\in I$ we have that $P_{\Phi}(t g)=s_{\infty}+t\alpha$. Moreover if $t\in \text{int}(I)$ then either $g$ is such that $\Delta_g$ is  cohomologous to the function $\alpha \tau$ or $tg$ is transient.
\item
For $t\in\R\backslash I$ the function $t\to P_{\Phi}(t g)$ varies analytically and $tg$ is positive recurrent.
\end{enumerate}
\end{prop}

\begin{proof}
We prove the two items in separate ways.

For (a), fix $t\in I$. Then Lemma \ref{lem:critvalue} implies that if $s<s_{\infty}+t\alpha$ then $P_{\overline{\sigma}}(\overline{-s\tau+t\Delta_g})=\infty$.  Since $t\in I$,
$P_{\overline{\sigma}}(\overline{-(s_{\infty}+t\alpha){\tau}+\Delta_g})\le 0$, so
by Lemma \ref{lem:indu-pre flow} we have that $P_{\Phi}(t g)=s_{\infty}+t\alpha$. Furthermore since $\overline{\tau}$ and $\overline{\Delta_g}$ are locally H\"{o}lder the function $t\to P_{\overline{\sigma}}(\overline{-s_{\infty}\tau+t(\alpha\tau+\Delta_g)})$ is analytic and convex for $t\in \text{int}(I)$ (see \cite{Sar03}). Thus either $P_{\overline{\sigma}}(\overline{-(s_{\infty}+t\alpha)\tau+t\Delta_g})<0$ for all $t\in  \text{int} I$ in which case $tg$ is transient, or $P_{\overline{\sigma}}(\overline{-(s_{\infty}+t\alpha)\tau+t\Delta_g})=0$ for all $t\in I$.

To complete the proof of (a), suppose that $\text{int}(I)\neq\emptyset$ and  $P(-\overline{(s_{\infty}+t\alpha)\tau+t\Delta_g})=0$ for all $t\in I$. Since $P_{\overline{\sigma}}(\overline{-s_\infty\tau+t(\alpha\tau+\Delta_g)})=0$ for all $t\in\text{int}(I)$ the associated Gibbs state $\mu_t$ has $\int\overline{\Delta_g-\alpha\tau}~{\rm d}\mu_t=0$. Thus $\mu_t$ is an equilibrium measure for $\overline{-(s_{\infty}+t\alpha)\tau+t\Delta_g}$ for all $t\in \text{int}(I)$. By \cite[Theorem 2.2.7]{mubook} this implies that $\overline{\Delta_g-\alpha\tau}$ is cohomologous to a constant which must be $0$. Thus $\Delta_g$ must be cohomologous to $\alpha\tau$.

For (b) we can follow the method from Proposition 6 in \cite{bi1}.
Note that the function $s\to P_{\overline{\sigma}}(\overline{-s\tau+t\Delta_g})$ is real analytic and decreasing for $s>s_{\infty}+t\alpha$.
So if $J\subseteq\R$ is an open interval for which $J\cap I=\emptyset$ then for $t\in J$ we can define $P_{\Phi}(tg)$ implicitly by
$$P_{\overline{\sigma}}(\overline{-P_{\Phi}(tg)\tau+t\Delta_g})=0.$$
If we let $\overline{\nu_t}$ denote the equilibrium measure for $\overline{-P_{\Phi}(tg)\tau+t\Delta_g}$ then it follows that
$$\frac{\partial}{\partial s}P_{\overline{\sigma}}(\overline{-s\tau-t\Delta_g})\left|_{s=P_{\Phi}(tg)}\right.=-\int\overline{\tau}~{\rm d}\mu_t<0.$$
Thus we can apply the implicit function theorem to show that $t\to P_{\Phi}(tg)$ is analytic on $J$. Hence $\nu=R(\mu_t)$ will be the equilibrium measure for $tg$ and so $tg$ is positive recurrent.
\end{proof}
Thus we have phase transitions if $I$ is an interval not equal to $\R$.

\begin{eje}
First of all we give an example of a potential  $g$ where $t\to P_{\Phi}(tg)$ is analytic for the whole of $\R$.
We define
$$\tau(x)=\log (n+2)-\log (n+1)\text{ if }x\in C_n$$
and suppose $g$ satisfies
$$\Delta_g(x)=-\log\log\log (n+2)+\log\log\log (n+1)\text{ if }x\in C_n.$$
This gives that
$$\overline{\tau}(x)=\log (n+1)\text{ if }x\in \overline C_n$$
and
$$\overline{\Delta_g}(x)=-\log\log\log (n+1)\text{ if }x\in \overline C_n.
$$
Using the notation above this means that $s_{\infty}=1$ and $\alpha=0$ (recall that $\alpha$ is defined in \eqref{eq:alpha}). For any $t\in\R$ we have that
$$P_{\overline{\sigma}}(-s_\infty\overline{\tau+t\Delta_g})=\log\left(\sum_{n=1}^{\infty}\frac{(\log\log(n+2))^{-t}}{n+2}\right)=\infty.$$
Thus we have that $I=\emptyset$ and by Proposition \ref{PT} the function $t\to P_{\Phi}(tg)$ is analytic.
\end{eje}

\begin{eje} \label{eg:2phase}
We now give an example with two phase transitions. We choose $K>2$  to satisfy $\sum_{n=0}^{\infty}((n+K)(\log(n+K))^2)^{-1}<\frac{1}{9}$. We suppose that
$$\tau(x)=\left\{\begin{array}{lll}\log 2&\text{ if }&x\in C_0 \\
\log(K)-\log 2&\text{ if }&x\in C_1\\
\log(K+n-1)-\log (K+n-2)&\text{ if }&x\in C_n\text{ for }n\geq 2\end{array}\right.
$$
and $g$ satisfies
$$\Delta_g(x)=\left\{\begin{array}{lll}\log(4/3)&\text{ if }&x\in C_0 \\
-\log\log(K)-\log (4/3)&\text{ if }&x\in C_1\\
-\log\log(K+n-1)+\log\log (K+n-2)&\text{ if }&x\in C_n\text{ for }n\geq 2\end{array}\right.
$$
This gives
$$\exp(\overline{-\tau+t\Delta_g}(x))=\frac{1}{2}\left(\frac{4}{3}\right)^t\text{ if }x\in C_0$$
and for $n\geq 2$
$$\exp(\overline{-\tau+t\Delta_g}(x))=\frac{1}{\log(K+n-2)}(\log (K+n-2))^{-t}\text{ if }x\in C_n.$$
As in the previous example we have that $s_{\infty}=1$ and $\alpha=0$. For $t\in\R$ we obtain
$$P_{\overline{\sigma}}(\overline{-s_\infty\tau+t\Delta_g})=\log\left(\frac{1}{2}\left(\frac{4}{3}\right)^t+\sum_{n=2}^{\infty}\frac{1}{ K+n-1}(\log (K+n-1))^{-t}\right).$$
Now for $t\leq 1$ this is divergent and for $t>1$ this is convergent. If we take $t=2$ then
$$P_{\overline{\sigma}}(\overline{-s_{\infty}\tau+2\Delta_g})=\log\left(\frac{8}{9}+\sum_{n=0}^{\infty}((n+K)(\log(n+K))^2)^{-1}\right)<0,$$
and so $\underline{t}=\inf\{t:P_{\overline{\sigma}}(-\overline{\tau+t\Delta_g}(x))\leq 0\}\in (1,2)$. Furthermore for $t=3$, we have that $\frac{1}{2}\left(\frac{4}{3}\right)^t>1$ and so $P_{\overline{\sigma}}(\overline{-s_{\infty}\tau+3\Delta_g})>0$. Thus
$\overline{t}=\sup\{t:P_{\overline{\sigma}}(-\overline{\tau+t\Delta_g})\leq 0\}\in (2,3)$. Therefore $I=[\underline{t},\overline{t}]$ and for $t\in I$ we have $P_{\Phi}(tg)=s_{\infty}=1$. There are phase transitions at $\underline{t}$ and $\overline{t}$ and outside $I$ the function $t\to P_{\Phi}(tg)$ varies analytically and is strictly greater than $s_{\infty}$.
\end{eje}

\subsection{Improving (or not) recurrence properties}

 In this sub-section, we discuss the idea that by suspending a system with a roof function not bounded away from zero, we can
speed up the return times improving the mixing properties and therefore obtaining better thermodynamics.   Indeed, we demonstrate that we can arrange roof functions and potentials so that any pair from the set $\{$positive recurrent, null recurrent, transient$\}$  can be produced with the first behaviour on the base and the second on the corresponding flow.
 Again the example we consider is the renewal flow $\Phi_R$, with roof function $\tau$
and potential $g: Y \to \R$.
We say that the recurrence properties of the potential $\Delta_g:\Sigma \to \R$   \emph{improve} if $\Delta_g$ is transient and $g$ is recurrent or
if $\Delta_g$ is null-recurrent and $g$ is positive recurrent. In fact it is possible for $g$ to be positive-recurrent, null-recurrent or transient for all behaviours of $\Delta_g$.

\begin{eje} \label{eg:better-worse}
Let $\tau\in \mathcal{R}$ be a roof function such that
\begin{equation*}
P_\sigma(-t \tau)=
\begin{cases}
\text{positive} & t < 1;\\
0 & t \geq 1.
\end{cases}
\end{equation*}
By virtue of Proposition \ref{prop:BI ana} we have that for $t<1$ the potential $-t \tau$ is positive recurrent and for $t >1$ the potential $-t \tau$ is transient. Moreover, for $t=1$ the potential can be positive recurrent, null-recurrent or transient (see Section~\ref{sec:Hof roof}  and \cite[Example 2]{Sar01a}).
 Consider the renewal flow $\Phi_R$ with roof function $\tau$.  The corresponding potential $\Delta_g = C \tau$ is such that
\[P_\Phi(tg)= \inf \{ q \in \R : P_\sigma(tC \tau - q \tau)  \leq 0\}.\]
Since $P_\Phi(tC \tau - q \tau)= P_\sigma((tC-q) \tau)$ we obtain
\[P_\sigma(tg)= tC+1.\]
The above, of course, could have been obtained from the variational principle for the flow.    In each of the cases below we are able to make our conclusion using Proposition~\ref{prop:ren-eq}.

\begin{enumerate}
\item Let us consider first the case in which the potential $-\tau$ is transient. In this case the potential $t\Delta_g -P_{\Phi}(g) \tau= -\tau$ is transient, therefore the potential $tg$ is transient for every $t \in \R$. However, if $t<-1/C$ the potential $t\Delta_g$ is positive recurrent,
whereas if $t\geq -1/C$ then $t\Delta_g$ is transient.
\item
Assume that $-\tau$ is null-recurrent with infinite measure $\nu$ such that $\int \tau~{\rm d}\nu=\infty$. In this case $t g$ is null recurrent for all $t\in\R$ whereas $t\Delta_g$ is positive recurrent if $t<-1/C$, null recurrent if $t=-1/C$ and transient if $
t>-1/C$.
\item Assume that  $-\tau$ is null-recurrent with infinite measure $\nu$ such that $\int \tau~{\rm d} \nu < \infty$ (see Lemma \ref{lem:MP} for an example of this). In this case $t g$ is positive recurrent for all $t\in\R$. However if $t<-1/C$ then the potential  $t\Delta_g$ is positive recurrent, if $t=-1/C$ then $t\Delta_g$ is null-recurrent and if $t>-1/C$ then $t\Delta_g$ is transient.

\end{enumerate}
\end{eje}

The simple example above shows that the recurrence properties and the thermodynamic formalism  can either improve or get worse by suspending with a roof function not bounded away from zero. In fact we have shown that all combinations of behaviour are possible except $g$ transient and $\Delta_g$ null recurrent. We finish this subsection with an example to show that this is also possible.  Observe that in this example the roof function is bounded away from zero.
\begin{eje}
We let $f\in\mathcal{R}$ be a positive function such that
\begin{equation*}
P_\sigma(-t f)=
\begin{cases}
\text{positive} & t < 1;\\
0 & t \geq 1
\end{cases}
\end{equation*}
and $-f$ is null recurrent. As in Proposition~\ref{prop:BI ana}, for $t>1$ the potential $-tf$ must be transient.  We let $\tau=f+1$ and let $\Delta_g=-f+1$ which is null recurrent with $P_\sigma(\Delta_g)=1$. We
then have that $P_{\sigma}(\Delta_g-\tau)=P_{\sigma}(-2f)=0$ and so $P_{\Phi}(g)=1$ with $g$ transient. Therefore $\Delta_g$ is null recurrent and $g$ is transient.
\label{eg:trans null}
\end{eje}

\section{Suspension flows over Manneville-Pomeau maps}
\label{sec:MP flow}
\subsection{ Manneville-Pomeau flows} \label{ssec:mpflows}
In this section we study suspension flows over a simple non-uniformly hyperbolic interval map, namely the Manneville-Pomeau map \cite{ManPom}.
We give the form studied in \cite{LivSaVai}. For $  \alpha> 0$ the map is defined by
\begin{equation*}
f(x) =
\begin{cases}
x(1 + 2^{\alpha}x^{\alpha}) & \text{ if }  x \in [0, 1/2),\\
2x - 1 & \text{ if }  x \in[1/2, 1).
\end{cases}
\end{equation*}
The pressure function of the potential $-\log |f'|$ satisfies the following (see  \cite{Lo,Sar01a}),
\begin{equation*}
P_f(-t \log |f'|)=
\begin{cases}
\text{strictly convex and real analytic} & \text{ if } t <1,\\
0 & \text{ if } t \geq 1.
\end{cases}
\end{equation*}
It is well known (see \cite{Lo,Sar01a}) that if $\alpha \in (0, 1)$ then  there exists an absolutely continuous invariant measure. This measure together with the Dirac delta at zero are the equilibrium measures for $-\log|f'|$.
If $\alpha >1$ then  there is no absolutely continuous invariant probability measure and the only equilibrium measure for $-\log|f'|$ is the Dirac delta at zero.  However, there exists an  infinite $f$-invariant measure $\nu$ absolutely continuous with respect to Lebesgue. This measure is such that (see \cite[p.849]{Nak00})
\begin{equation} \label{eq:fini}
\int \log|f'|~{\rm d} \nu < \infty.
\end{equation}

If we remove the parabolic fixed point at zero and its preimages, then the (non-compact) dynamical system that is left can be coded by the renewal shift (see \cite{Sar01a}). More precisely, if  we denote by $\Omega = [0,1] \setminus \cup_{n=0}^{\infty} f^{-n}(0)$, then the map $f$ restricted to $\Omega$ can be coded by the renewal shift.

We define the \emph{Manneville-Pomeau flow}, that we denote by $\Phi_{mp}$,  as the suspension semi-flow with base $f(x)$ and roof-function $\log |f'|$. We denote by $Y$ its phase space. This flow has a singularity at $(0,0)$ and there exists an atomic invariant measure supported on it. If we remove the point $(0,0)$ and all its pre-images, then the non-compact semi-flow that is left can be coded as a renewal flow with roof function the symbolic representation of $\log |f'|$. We denote this renewal flow by $\Phi_R$. The set of invariant measures for the Manneville-Pomeau flow is in one to one correspondence with the set of invariant measures for  the renewal flow, denoted by $\M_{\Phi_R}$, together with the atomic measure supported at $(0,0)$.

\begin{rem}
Using the standard inducing scheme on the interval $[1/2, 1)$, yields an expression for the Manneville-Pomeau flow to which the main theorem of \cite{BalVal05} applies directly.  This means that the associated flow has exponential decay of correlations for $C^1$ observables for the SRB measure on the flow.  See  \cite[Section 1]{BalVal05} and \cite[Section 4]{Pol99} for details.
We might expect similar results for equilibrium states for Rovella flows, see \cite{PacTod10}, possibly requiring the subexponential decay results of Melbourne \cite{m2}.  We also remark that Holland, Nicol and T\"or\"ok have studied Extreme Value Theory for semi-flows built over the Manneville-Pomeau map in \cite{hnt}.
\end{rem}

Let $g:Y \to \R$ be a bounded potential and consider the map $\Delta_g: [0,1] \to \R$ defined by
\[ \Delta_g(x)= \int_0^{\log|f'(x)|} g(t,x) dt.\]
Denote by $\Delta_g^r$ the symbolic representation of $\Delta_g(x)$  in the renewal shift. We will develop a thermodynamic formalism for the following class of potentials:
\[ \mathcal{MP}:= \left\{g:Y \to \R : g \text{ is bounded and }  \Delta_g^r(x) \in \mathcal{R}  \right\}.\]
We define the pressure of a potential $g \in \mathcal{MP}$ by
\begin{equation*}
P_{\Phi_{mp}}(g):=\sup \left\{h_f(\mu) + \int \ g~{\rm d} \mu : \mu \in \mathcal{E}_{mp}  \right\},
\end{equation*}
where $\mathcal{E}_{mp}$ denotes the set of ergodic $\Phi_{mp}$-invariant probability measures. We stress that  $\mathcal{E}_{mp}$ is in one to one correspondence with the set  $\mathcal{E}_{\Phi_R} \cup \delta_{(0,0)}$.

\begin{lema}\label{lem:MP}
Every Manneville-Pomeau flow has entropy equal to one and there exists a unique measure of maximal entropy.
\end{lema}

\begin{proof}
Recall that $h(\Phi_{mp})= \inf \{s \in \R : P(-s \log |f'|) \leq 0 \}. $ Since
\[P_f(-\log|f'|)=0,  \]
and for every $t<1$  we have that $ P_f(-t\log|f'|)>0$ the entropy of the flow is equal to $1$. The potential $-\log|f'|$ is recurrent and  because of the property stated in equation \eqref{eq:fini} the flow has a measure of maximal entropy. It is unique since potentials of summable variations over countable Markov shifts have at most one equilibrium measure \cite{busa} and the atomic measure supported at $(0,0)$ has zero entropy.
\end{proof}
Note that if $\alpha\geq 1$ then this measure of maximal entropy will be finite for the flow but will project to an infinite invariant measure for the Manneville-Pomeau map.

In our next result we discuss the existence and uniqueness of equilibrium measures for potentials in $\mathcal{MP}$. Denote by $\log^r|f'|$ the symbolic representation of $\log|f'|$ in the renewal shift. First note that it follows from Proposition \ref{prop:ren-eq} that the equation $P(\Delta_g^r -s \log^r|f'|)=0$, always has a root, that we denote by  $P_{\Phi_{mp}}^r(g)$.

\begin{prop}
Let $g \in \mathcal{MP}$.
\begin{enumerate}[label=({\alph*}),  itemsep=0.0mm, topsep=0.0mm, leftmargin=7mm]
\item If $\Delta_g^r -P_{\Phi_{mp}}^r(g) \log^r|f'|$ is positive recurrent then there exists an equilibrium measure for $g$.  Moreover,
\begin{enumerate}[label=({\roman*}),  itemsep=0.0mm, topsep=0.0mm, leftmargin=3mm]
\item If $P_{\Phi_{mp}}^r(g) \neq \int g~{\rm d}\delta_{(0,0)}$ then the equilibrium measure for $g$ is unique.
\item  If $P_{\Phi_{mp}}^r(g) = \int g~{\rm d}\delta_{(0,0)}$ then there are exactly two equilibrium measures for $g$.
\end{enumerate}
\item If $\Delta_g^r -P_{\Phi_{mp}}^r(g) \log^r|f'|$ is null recurrent with infinite measure $\nu$ and we have that $\log^r|f'| \in L^{1}(\nu)$  then there exists an equilibrium measure for $g$.  Moreover,
\begin{enumerate}
\item If $P_{\Phi_{mp}}^r(g) \neq \int g~{\rm d}\delta_{(0,0)}$ then the equilibrium measure for $g$ is unique.
\item  If $P_{\Phi_{mp}}^r(g) = \int g~{\rm d}\delta_{(0,0)}$ then there are exactly two equilibrium measures for $g$.
\end{enumerate}
\item Assume that $\Delta_g^r -P_{\Phi_{mp}}^r(g) \log^r|f'|$ is null recurrent with infinite measure $\nu$ and $\log^r|f'| \notin L^{1}(\nu)$. If  $P_{\Phi_{mp}}^r(g) > \int g~{\rm d}\delta_{(0,0)}$ then there is no equilibrium measure for $g$. On the other hand, if $P_{\Phi_{mp}}^r(g) < \int g~{\rm d}\delta_{(0,0)}$ then $\delta_{(0,0)}$ is the unique equilibrium measure for $g$.
\item  Assume that $\Delta_g^r -P_{\Phi_{mp}}^r(g) \log^r|f'|$ is transient. If  $P_{\Phi_{mp}}^r(g) > \int g~{\rm d}\delta_{(0,0)}$ then there is no equilibrium measure for $g$. On the other hand, if $P_{\Phi_{mp}}^r(g) < \int g~{\rm d}\delta_{(0,0)}$ then $\delta_{(0,0)}$ is the unique equilibrium measure for $g$.
\end{enumerate}
\end{prop}

\begin{proof}
The proof follows directly from Proposition \ref{prop:ren-eq} and the observation that
\[P_{\Phi_{mp}}(g)=  \max \left\{ P_{\Phi_{mp}}^r(g), \int g~{\rm d}\delta_{(0,0)} \right\}.\]
\end{proof}

\begin{rem}
It is possible to use Proposition \ref{PT} to obtain potentials where the pressure function for the Manneville-Pomeau flow has phase transitions. If we consider the induced potential $\overline{\tau}$ we have that $s_{\infty}=\frac{\alpha}{\alpha+1}$. We now take a negative function $g$ where $\overline{\Delta{g}}$ satisfies
$$\lim_{n\to\infty}\frac{\sup_{x\in C_n}\{\overline{\Delta_g}(x)\}}{\inf_{x\in C_n}\{\overline{\tau}(x)\}}=0$$
and
$$\lim_{n\to\infty} \frac{\inf_{x\in C_n}\{\overline{-\Delta_g}(x)\}}{\log\log n}=\infty.$$
This means that
for all $t>0$ we have $P_{\overline{\sigma}}(\overline{-s_\infty\tau+t\Delta_g})<\infty$ and there will exist $t^*\in\R$ such that $P_{\overline{\sigma}}(\overline{-s_\infty\tau+t\Delta_g})=0$ and for all $t\geq t^*$, $P_{\overline{\sigma}}(\overline{-s_\infty\tau+t\overline{\Delta_g}})<0$. Thus $t\to P_{\Phi}(tg)$ will have a (positive entropy) phase transition at $t^*$.
\end{rem}

\end{document}